\documentclass[10pt,twoside, a4paper, english, reqno]{amsart}
\usepackage{graphicx, amsmath, varioref, amscd, amssymb,color, bm, stmaryrd, amsthm}
\usepackage{fancybox}
\usepackage[pdftex,colorlinks=false]{hyperref}

\numberwithin{equation}{section}
\allowdisplaybreaks

\newtheorem{theorem}{Theorem}[section]
\newtheorem{lemma}[theorem]{Lemma}

\theoremstyle{definition}
\newtheorem{definition}[theorem]{Definition}

\theoremstyle{remark}
\newtheorem{remark}[theorem]{Remark}

\newcommand{\Div}{\operatorname{div}}

\newcommand{\Curl}{\operatorname{curl}}
\newcommand{\Grad}{\nabla}
\newcommand{\vr}{\varrho}

\newcommand{\vc}[1]{{\bm{#1}}}

\newcommand{\weak}{\rightharpoonup}

\newcommand{\norm}[1]{\left\Vert#1\right\Vert}
\newcommand{\abs}[1]{\left|#1\right|}

\newcommand{\R}{\mathbb{R}}
\newcommand{\N}{\mathbb{N}}

\newcommand{\Om}{\ensuremath{\Omega}}

\newcommand{\vt}{\theta}
\newcommand{\vvt}{\vartheta}
\newcommand{\dott}{\, \cdot\,}

\begin{document}
\title[Operator splitting for active scalar equations]{Operator splitting 
for well-posed \\ active scalar equations}

\author[Holden]{Helge Holden}
\address[Helge Holden]{\newline
    Department of Mathematical Sciences\newline
    Norwegian University of Science and Technology\newline
    NO--7491 Trondheim, Norway\newline
{\rm and} \newline
  Centre of Mathematics for Applications\newline
University of Oslo\newline
  P.O.\ Box 1053, Blindern,
  NO--0316 Oslo, Norway }
\email[]{\href{holden@math.ntnu.no}{holden@math.ntnu.no}}
\urladdr{\href{http://www.math.ntnu.no/~holden}{www.math.ntnu.no/\~{}holden}}

\author[Karlsen]{Kenneth H.~Karlsen}
\address[Kenneth H. Karlsen]{\newline
   Centre of Mathematics for Applications\newline
 University of Oslo \newline
  P.O.\ Box 1053, Blindern,
  NO--0316 Oslo, Norway}
\email[]{kennethk@math.uio.no}
\urladdr{http://www.kkarlsen.com}

\author[Karper]{Trygve K. Karper}
\address[Trygve K. Karper]{\newline
University of Maryland, CSCAMM\newline
4146 CSIC Building \#406 
Paint Branch Drive \newline
College Park, MD 20742-3289,
USA
}

\date{\today}




\thanks{Supported in part by the Research Council of Norway.}

\maketitle

\begin{abstract}
We analyze operator splitting methods applied to scalar equations 
with a nonlinear advection operator, and a linear (local or nonlocal) 
diffusion operator or a linear dispersion operator. 
The advection velocity is determined from the scalar 
unknown itself and hence the equations are so-called active scalar equations.  
Examples are provided by the surface quasi-geostrophic 
and aggregation equations. In addition, Burgers-type equations 
with fractional diffusion as well as the KdV and Kawahara 
equations are covered.  Our main result is that the 
Godunov and Strang splitting methods converge with the expected rates 
provided the initial data is sufficiently regular.
\end{abstract}


\section{Introduction and main results}
We  consider operator splitting applied to a class of evolution equations having a 
nonlinear ``transport part", and a linear (local or nonlocal) ``diffusion part" or 
a linear ``dispersion part".   These equations are 
posed on $\R^N$ for $N=1, 2, 3$, and are of the form
\begin{equation}\label{eq:eq}
	u_t + \Div \left(u\,\vc{v}(u)\right) = A(u),
\end{equation}
where $u(t,x)$ is a scalar function and 
$\Div$ is the spatial divergence operator.

The vector-valued operator $\vc{v}(\dott)$ and 
the real-valued operator $A(\dott)$ are linear and satisfy 
a set of hypotheses (given in Definition \ref{def:operators} below).
These hypotheses are met, for example, by the popular fractional quasi-geostrophic 
\cite{Caffarelli1,Constantin1,Constantin2,Constantin4,Cordoba,Dong,Kiselev2} 
and aggregation equations 
\cite{Bertozzi-Carrillo,Bertozzi-Laurent1,DongAG1,DongAG2,Mogilner1,Bertozzi1,Bertozzi2}:
\begin{align}
	u_t +   \Div (u \,\vc{v}(u)) + (-\Delta)^{\alpha/2}u &= 0, 
	\quad \vc{v}(u) = \Curl (-\Delta)^{-\beta/2}u, \label{eq:qg}\\ 
	u_t + \Div \left(u\, \vc{v}(u)\right) + (-\Delta)^{\alpha/2}u &= 0, \label{eq:ag}
	\quad \vc{v}(u) = \nabla \Phi \star u,
\end{align}
where $\beta,\alpha \geq 1$. Both these equations are 
paramount examples of so-called \emph{active scalar equations}, 
that is, equations in which the advection velocity $\vc{v}(u)$ is 
determined by the scalar unknown $u$ itself. 
In three dimensions, the general formulation \eqref{eq:eq} encompasses 
also the active scalar equation \cite{Constantin3}
\begin{equation*}
	u_t + \Div(u\,\vc{v}(u)) - \Delta u = 0, 
	\quad \vc{v}(u) = \Div \mathbb{T} (u),
\end{equation*}
for which well-posedness was established recently in \cite{Friedlander}.
Here, $\mathbb{T}$ is a matrix of Calderon--Zygmund operators such 
that $\Div \vc{v}(u) = \Div \Div \mathbb{T}(u) = 0$. This equation 
is a generalized 3D version of the quasi-geostrophic 
equation \eqref{eq:qg} ($\alpha = 2$).

In one dimension, our linearity requirement on $\vc{v}(\dott)$ limits 
the type of equations we can consider to those of Burgers type, such as
\begin{align*}
	u_t + (u^2)_x &= u_{xxx}, & (\text{{\sc{KdV}}}), \\
	u_t + (u^2)_x &= u_{xx},& (\text{viscous Burgers}), \\
	u_t + (u^2)_x &= -u_{xxx} + u_{xxxxx}, & (\text{Kawahara}).
\end{align*}

Since the class of equations studied herein covers a variety of physical models, 
we will give a proper discussion of applications at the end of the paper.

The main topic of the present paper is analysis of operator splitting methods 
for constructing approximate solutions to \eqref{eq:eq}.  
The tag ``operator splitting" refers to  the classical 
idea of constructing  numerical methods for complicated partial 
differential equations by reducing the original equations to a 
series of equations with simpler structure, each of which 
can be handled by some efficient and tailored numerical method. 
We do not survey the literature on operator splitting here, referring the reader instead 
to the bibliography in \cite{Holden:book}.

The purpose of this paper is to  prove, in the context of 
active scalar equations \eqref{eq:eq}, the well-posedness 
and convergence rates for two frequently used operator splitting methods.
Both methods are based on applying repeatedly the transport 
operator and the diffusion/dispersion operator $A$
in separate steps.  This splitting is very reasonable since one can then 
use ``hyperbolic" numerical methods in the 
transport step and ``Fourier space" 
methods in the diffusion/dispersion step.  
The methods are in the literature referred 
to as Godunov and Strang splitting and are widely used 
 for both numerical computations and  analysis. 
The reader can consult \cite{Holden:book} for a recent survey 
of theory and applications; see also \cite{Majda:2002kx} 
for analysis of splitting algorithms for the incompressible 
Navier-Stokes equations. 
 
Let us now discuss our splitting methods in more detail. 
For this purpose, we first recast \eqref{eq:eq} 
in the form
\begin{equation*}
	u_t = C(u), \quad C(u) = A(u) + B(u),
\end{equation*}
where we have introduced the operator $B(u) :=-\Div (u \vc{v}(u))$.
We can then construct two solution operators 
$\Phi_A$ and $\Phi_B$ associated with the 
abstract ordinary differential equations
\begin{equation*}
	\begin{split}
		\partial_t\Phi_A(t,u_0) &= A(\Phi_A(t,u_0)), 
		\quad \Phi_A(0,u_0)= u_0, \\
		\partial_t\Phi_B(t,u_0) &= B(\Phi_B(t,u_0)), 
		\quad \Phi_B(0,u_0)= u_0.
	\end{split}
\end{equation*}
The first method we will consider, Godunov splitting, 
is defined as follows.
For $\Delta t> 0$ given,  construct a sequence 
$\{u^n, u^{n+1/2}\}_{n=1}^{\lfloor T/\Delta t\rfloor}$ of approximate solutions 
to \eqref{eq:eq} by the following procedure:
Let $u^0 = u_0$ and determine inductively
\begin{equation}\label{defI:godunov}
	u^{n+1/2} = \Phi_B(\Delta t,u^{n}),\quad    
	u^{n+1} = \Phi_A(\Delta t, u^{n+1/2}), 
	\quad n=0, \ldots, \lfloor T/\Delta t\rfloor-1,
\end{equation}
where $\lfloor z \rfloor$ gives the greatest integer less than or equal to $z$.

For Godunov splitting we prove that it is well-posed
and that it convergences linearly in $\Delta t$. Specifically, 
we prove the following theorem.
\begin{theorem}\label{thm:godunov}
Let $T>0$ be given and assume $u_0 \in H^{k}$ with $6 \leq k \in \mathbb{N}$.
Then, for $\Delta t>0$ sufficiently small we have the following:	
\begin{enumerate}
	\item The Godunov method \eqref{defI:godunov} is well-defined with
	\begin{equation*}
		\|u^n\|_{H^k} \leq C, \quad n=1, \ldots, \lfloor T/\Delta t \rfloor.
	\end{equation*}
	\item The error satisfies
	\begin{equation}\label{err:godunov}
		\|u^n - u(n\Delta t)\|_{H^{k-\max\{\alpha,2\}}} 
		\leq C\|u_0\|_{H^{k}} \Delta t,
	\end{equation}
	where $\alpha$ is the highest number 
	of derivatives occuring in $A$.
\end{enumerate}

\end{theorem}

The other method we consider, Strang splitting, is 
defined as follows: For $\Delta t> 0$ given, construct a sequence 
$\{u^n, u^{n+1/4}, u^{n+3/4}\}_{n=1}^{\lfloor T/\Delta t\rfloor}$ 
of approximate solutions to \eqref{eq:eq} by the following procedure:
Let $u^0 = u_0$ and determine inductively,
for $n=0, \ldots, \lfloor T/\Delta t\rfloor-1$,
\begin{equation}\label{defI:strang}
	\begin{split}
		u^{n+1/4} &= \Phi_B(\frac12\Delta t,u^{n}), \quad
		u^{n+3/4} = \Phi_A(\Delta t, u^{n+1/4}),  \\
		u^{n+1}   &= \Phi_B(\frac12\Delta t, u^{n+3/4}).
	\end{split}
\end{equation}

For the Strang splitting algorithm we prove well-posedness and second order 
convergence, provided the initial data are sufficiently regular.
\begin{theorem}\label{thm:strang}
Let $T>0$ be given and assume $u_0 \in H^{k}$ with $6 \leq k \in \mathbb{N}$.
Then, for $\Delta t>0$ sufficiently small we have the following:	
\begin{enumerate}
	\item The Strang method \eqref{defI:strang} is well-defined with
	\begin{equation*}
		\|u^n\|_{H^k} \leq C, \quad n=1, \ldots, \lfloor T/\Delta t \rfloor.
	\end{equation*}
	\item The error satisfies \eqref{err:godunov} and
	\begin{equation}\label{err:strang}
		\|u^n - u(n\Delta t)\|_{H^{k-3\max\{\alpha,1\}}} \leq C (\Delta t)^2.
	\end{equation}
\end{enumerate}
The $\alpha$ occurring in \eqref{err:strang} is the highest number of derivatives in $A$.
\end{theorem}

The approach leading to Theorems \ref{thm:godunov} and \ref{thm:strang} 
utilizes the analysis framework put forth in the recent paper \cite{Holden:tao} 
for the KdV equation.  This framework works with 
a definition of the Godunov method given in terms of a specific extension of the 
splitting solution $\{u^n, u^{n+1/2}\}_n$ to all of $[0,T]$.
An adaption of this extension that also applies to Strang splitting 
was provided in \cite{Holden:karper}. This extension, which is 
different from the one used in \cite{Holden:tao} for Strang 
splitting, will be  employed herein.  This allows us to treat in a unified manner 
a rather general class of equations,  covering those treated 
in \cite{Holden:tao}, \cite{Holden:karper}, and \cite{Holden:lubich}, 
as well many additional equations not treated in these papers.
At variance with \cite{Holden:karper}, we do not require the divergence of the 
velocity field $\vc{v}(\dott)$ to be zero, thereby enlarging significantly the class 
of equations that can be handled.

The remaining part of this paper is organized as follows: 
Section \ref{sec:prelim} collects some preliminary 
results needed later on. The Godunov and Strang splitting methods are analyzed 
in Sections \ref{sec:godunov-split} and \ref{sec:strang-split}, respectively.

\section{Preliminary existence and regularity results}\label{sec:prelim}

\subsection{Notation and two technical estimates}
We will use $L^p$ to denote the Lebesgue space of integrable function on $\R^N$
with exponent $p$.
If $l$
denotes a N-dimensional multi-index, i.e., 
$l=(l_1,\ldots, l_N)$, $l_j\in\mathbb{N}_0$, we write
\begin{equation*}
  D^l f=\Grad^{l} f =\frac{\partial^{\abs{l}}f}{\partial
    x_1^{l_1}\cdots\partial {x_N}^{l_N}}, \qquad \abs{l}=l_1+\cdots+l_N,
\end{equation*}
to denote any derivative of the $l$th order.
If $\ell\in \mathbb{N}$, we let
\begin{equation*}
	\Grad^{\ell} f = \{\Grad^{l} f\mid |l| = \ell\}
\end{equation*}
and 
\begin{equation*}
	\Grad^{\ell} f:\Grad^{\ell} g = 
	\sum_{\abs{l}\le \ell}  \Grad^{l} f\, \Grad^{l} g.
\end{equation*}
We will be working the Sobolev spaces
\begin{equation*}
  H^k=\{f\in\mathcal S' \mid (1+\abs{\xi}^2)^{k/2}\mathcal F(
  f(\xi))\in L^2\},
\end{equation*}
(where $\mathcal S'$ denotes the set of tempered distributions) 
and $\mathcal F$  denotes Fourier transform. If $k$ is a natural number, 
$H^k$ is the standard Sobolev space with inner product and norm given by
\begin{equation*}
  \langle f,g\rangle_{H^k} 
=\sum_{\ell = 0}^k\langle \Grad^\ell f,\Grad^\ell g\rangle_{L^2},
\quad \norm{f}_{H^k}=  \langle f,f\rangle_{H^k}^{1/2},
\end{equation*}
where we have introduced
\begin{equation*}
	\langle \Grad^\ell f,\Grad^\ell g\rangle_{L^2}
	= \sum_{\substack{l\\ \abs{l} = \ell}} \langle D^l f,D^l g\rangle_{L^2}.
\end{equation*} 

Throughout the paper we will strongly rely on the following two technical lemmas.
The validity of these estimates are the primary reasons for the 
requirements on $\vc{v}$ given in Definition \ref{def:operators}.
Their proofs are straightforward, but somewhat tedious. 
For this reason, proofs are deferred to the appendix.
\begin{lemma}\label{lem:nonlineardiv}
Let $k\geq 6$. Then
\begin{equation*}
	\sum_{s=0}^k\left|\int_{\R^N} 
	\Grad^s(\Div (f \vc{v}(f))):\Grad^s f~  dx\right| \leq C\|f\|_{H^{k-2}}\|f\|_{H^k}^2,
	\quad  f \in H^k.
\end{equation*}	
\end{lemma}
\begin{lemma}\label{lem:lineardiv}
Let $k\geq 4$. Then the following estimates hold
\begin{align}\label{eq:ldiv1}
	\sum_{s=0}^k\left|\int_{\R^N} \Grad^s
	\Div \left(f \vc{v}(g)\right):\Grad^s f~dx\right| & 
	\leq C\|g\|_{H^k}\|f\|_{H^k}^2, \quad f,g\in H^k, \\
	\label{eq:ldiv2}
	\sum_{s=0}^k\left|\int_{\R^N}\Grad^s
	\Div \left(g\vc{v}(f)\right):\Grad^s f~dx\right| &\leq C\|g\|_{H^{k+1}}\|f\|_{H^k}^2, 
	\quad f\in H^k,\, g\in H^{k+1}.
\end{align}
\end{lemma}

\subsection{Existence and regularity results}
Our splitting methods are  based 
on alternately solving 
the equations
\begin{equation}\label{eq:lineareq}
	u_t = A(u), \quad u|_{t=0} = u_0,
\end{equation}
\begin{equation}\label{eq:transporteq}
	u_t + \Div (u \,\vc{v}(u)) = 0, \qquad u|_{t =0} = u_0.
\end{equation}
To analyze the methods we will need to have some control
on the behavior of solutions to each of
the two equations separately, at least on short time 
intervals.
Before we start discussing available results, let us 
state the specific properties we will require of the operators $\vc{v}$ and 
$A$.

\begin{definition}\label{def:operators}
We say that the operators $\vc{v}, A$ are admissible provided:
\begin{enumerate}
	\item $A$ and $\vc{v}$ are linear operators 
	satisfying the commutative property
	$$
	A(v_i(\dott)) = v_i(A(\dott)), \quad  i=1,\ldots, N,
	$$
	where $v_i$ is the $i$th component of $\vc{v}$.
	\item $\vc{v}:L^p\to L^p$ is bounded 
	for any $p<\infty$, and, if $N \geq 2$, 
	$$
	\|\!\Div \vc{v}(u)\|_{L^p} \leq 
	C \|u\|_{L^p}, \quad N \geq 2.
	$$
	\item $A$ is a differential operator; specifically, there is a positive 
	integer $\alpha$ such that $A\colon W^{\alpha + k, p} \rightarrow W^{k,p}$, 
	is bounded for all $k<\infty$ and $p < \infty$. 
	\item $A$ is either conservative or diffusive:
	$$
	\int_{\R^N} A(u)u~dx \leq 0, \quad  u \in  W^{\alpha + k, p}.  
	$$
	\item $A$ satisfies 
	the following commutator estimate
	\begin{equation*}
		\|A(fg)-fA(g)-gA(f)\|_{H^k} \leq 
		C\|f\|_{H^{k+ \max\{\alpha,2\}-1}}\|g\|_{H^{k+ \max\{\alpha,2\}-1}},
	\end{equation*}
	for all $f$, $g\in H^{k+ \max\{\alpha,2\}-1}$ for $k \geq 3$.
	\item The equation \eqref{eq:eq} is well-posed in the sense that for
	any given finite $T>0$ and initial data $u_0 \in H^k$, there is a solution 
	of \eqref{eq:eq} satisfying
	$$
	u \in C([0,T];H^k).
	$$
\end{enumerate}
\end{definition}

The first equation \eqref{eq:lineareq} is a linear equation with constant 
coefficients ((1) in Definition \ref{def:operators})
that preserves (or diffuses) all Sobolev norms 
over time. From this, the following lemma follows readily.

\begin{lemma}\label{lem:global}
If $u_0 \in H^{k}$ for some $\alpha \leq k<\infty$, then there is 
a unique solution $u \in C([0,T];H^k)$ of \eqref{eq:lineareq}.
\end{lemma}

The nonlinear equation \eqref{eq:transporteq} is of hyperbolic type.
Typically, the best available existence results for this 
type of equations are local in time existence of smooth solutions. 
\begin{lemma}\label{lem:local}
Let $u_0 \in H^k$ for some $k \geq 4$. There exists a time $T>0$ and 
a function $u \in C([0,T);H^k)\cap C^1([0,T);H^{k-1})$ such that 
$u$ is a unique solution of \eqref{eq:transporteq} on $[0,T)$. Moreover, 
if $T$ is the maximal time of existence for \eqref{eq:transporteq}, then
\begin{equation*}
	\lim_{t \rightarrow T}\|u(t)\|_{H^k} = \infty.
\end{equation*}
\end{lemma}

\begin{proof}
We will prove the local in time existence by demonstrating compactness
of solutions to the approximation scheme:
\begin{equation}\label{eq:apscheme}
	u^m_t + \vc{v}(u^{m-1})\cdot \Grad u^{m} 
	+ u^{m}\Div \vc{v}(u^{m-1}) = 0, \qquad u^m|_{t=0} = u_0.
\end{equation}	
For this purpose, we let $u^0 = u_0$ and 
sequentially determine the sequence $\{u^m\}_{m=1}^{\infty}$ 
as the solutions to the approximation scheme \eqref{eq:apscheme}. 
Now, for each given $u^{m-1}$, the requirements of Definition \ref{def:operators} 
yields $\vc{v}(u^{m-1})$, $\Div \vc{v}(u^{m-1}) \in C([0,T]; H^k)$. Thus, 
\eqref{eq:apscheme} is a linear transport 
equation with smooth coefficients and consequently 
admits a smooth solution for all times. 

Let us now calculate the $H^k$ norm of $u^m$. 
To achieve this, we apply $\Grad^s$ 
to \eqref{eq:apscheme}, multiply with 
$\Grad^s u^m$, sum over all $s=0, \ldots, k$, 
and integrate to obtain
\begin{equation*}
	\begin{split}
		\partial_t \frac{1}{2}\|u^{m}(t)\|_{H^k}^2 
		&= -\sum_{s=0}^k \int_{\R^N}\big( \Grad^s 
		\Div (u^m \vc{v}(u^{m-1})):\Grad^s u^m\big)(x,t)~dx \\
		&\leq C\|u^{m-1}(t)\|_{H^k}\|u^{m}(t)\|_{H^k}^2,
	\end{split}
\end{equation*}
where the last inequality is an application of Lemma \ref{lem:lineardiv}. Applying the 
Gronwall  inequality to the previous inequality yields
\begin{equation}\label{eq:constC}
	\|u^m\|_{L^\infty([0,T];H^k)} 
	\leq e^{CT\|u^{m-1}\|_{L^\infty([0,T];H^k)}}\|u_0\|_{H^k}.
\end{equation}
Next, we fix 
$$
T \leq \frac{\log{2}}{2 C\|u_0\|_{H^k}}, 
$$
where the constant $C$ is the one appearing in \eqref{eq:constC}. Using the bound on $T$
in \eqref{eq:constC} and iterating the resulting inequality (starting from $m=1$), we obtain
\begin{equation*}
	\|u^m\|_{L^\infty([0,T];H^k)} \leq 2\|u_0\|_{H^k}.
\end{equation*}

Next, let us calculate the $H^{k-1}$ norm of $u^m_t$. By direct calculation, 
\begin{equation*}
	\begin{split}
		\|u^m_t(t)\|_{H^{k-1}}
		&= \left|\sum_{s=0}^{k-1}\int_{\R^N}\big(\Grad^s 
		\Div (u^m \vc{v}(u^{m-1})):\Grad^s u_t^m\big)(x,t)~dx\right| \\
		&\leq C\|u^m(t)\|_{H^{k}}\|u^{m-1}(t)\|_{H^k}\|u^m_t(t)\|_{H^{k-1}}.
	\end{split}
\end{equation*}

At this point, we have proved that $u^m \in C([0,T);H^k)
\cap C^1([0,T);H^{k-1})$, independently of $m$. Hence, we can assert 
the existence of a function $u \in  C([0,T);H^k)\cap C^1([0,T);H^{k-1})$
such that 
$$
u^m \weak u \quad \text{in } C([0,T);H^k)\cap C^1([0,T);H^{k-1}),
$$
as $m \rightarrow \infty$, where the convergence 
might take place along a subsequence. Compact Sobolev embedding then tells us that 
\begin{equation*}
	u^m \rightarrow u, \quad \text{in } C([0,T);H^{k-1}),
\end{equation*}
again along a subsequence. Note that 
this is not sufficient to pass to the limit in \eqref{eq:apscheme}. 
Indeed, we need that the whole sequence converges.
To prove this, we let $w^m = u^m - u^{m-1}$ and observe that
\begin{equation*}
	w^m_t + \Div (\vc{v}(u^{m-1})w^m) + \Div(\vc{v}(w^{m-1})u^{m-1}) = 0. 
\end{equation*}
It follows that
\begin{equation*}
	\begin{split}
		\frac{d}{dt}\frac{1}{2}\|w^m\|_{L^2}^2 &= -\int \frac{1}{2}|w^m|^2\Div \vc{v}(u^{m-1}) + \Div\left(\vc{v}(w^{m-1})u^{m-1}\right)  w^m~dx \\
		&\leq  C\|u^{m-1}\|_{H^k}\left(\|w^m\|_{L^2}^2 + \|w^{m-1}\|_{L^2}^2 \right),
	\end{split}
\end{equation*}
where we have used the requirements on $\vc{v}(\cdot)$ (Definition \ref{def:operators}), 
Sobolev embedding, and the Cauchy inequality. An application of 
the Gronwall inequality, using that $w^m(0) = 0$, we obtain
\begin{equation*}
	\begin{split}
			\|w^m\|_{L^2}(t) \leq Ce^{CT} \int_0^t \|w^{m-1}(s)\|_{L^2}~ds 
			&\leq \frac{\left(tCe^{CT}\right)^m}{m!}\sup_{t\in (0,T)}\|w^1\|_{L^2}\\
			&\leq \frac{C^m}{m!}\sup_{t\in (0,T)}\|w^1\|_{L^2} \overset{m \rightarrow \infty}{\rightarrow}0.
	\end{split}
\end{equation*}
From this we conclude that $u^m$ is a Cauchy sequence in $L^2$ and hence 
that the entire sequence converges.
We can now pass to the limit in \eqref{eq:apscheme}
to conclude the existence part of the lemma. Uniqueness is an immediate 
consequence of the regularity of the solution (i.e subtracting one solution from 
a possibly different solution, adding and subtracting, and using the regularity).

The blow-up at maximal time of existence follows from $u$ being continuous in time.
\end{proof}

\section{Godunov splitting (proof of Theorem \ref{thm:godunov})}\label{sec:godunov-split}
To prove Theorem \ref{thm:godunov}, we will utilize the 
analysis framework put forth in \cite{Holden:tao}.
To this end, we introduce an extension of
the splitting solution $\{u^n, u^{n+1/2}\}_n$ to all of $[0,T]$. 
The definition is posed on the two-dimensional time domain
\begin{equation*}
	\Om_{\Delta t} =  
	\bigcup_{n=0}^{\lfloor T/\Delta t\rfloor-1}[t_n, t_{n+1}]\times [t_n, t_{n+1}],
\end{equation*}
and goes as follows:

\begin{definition}[Godunov splitting]\label{def:godunov}
For $\Delta t > 0$ given, we say that $\vvt$ is the Godunov 
splitting approximation to \eqref{eq:eq}
provided that $\vvt$ satisfies
\begin{equation*}\label{eq:godunov-2}
	\begin{split}
		\vvt(0,0) &= \vt_0, \\
		\vvt_t(t,t_n) &= B(\vvt(t,t_n)), \quad t \in (t_n, t_{n+1}], \\
		\vvt_\tau(t, \tau) &=  A(\vvt(t, \tau)), \quad (t,\tau) 
		\in [t_n, t_{n+1}]\times (t_n,t_{n+1}].
	\end{split}
\end{equation*}
\end{definition}
Observe that
\begin{equation*}
 \vvt(t_n,t_n)=u^n, \quad  n=0, \ldots, \lfloor T/\Delta t\rfloor-1.
\end{equation*}
Thus, $\vvt(t,t)$ is indeed an extension of $u^n$ to all of $[0,T]$.

To measure the error, we will use the function
\begin{equation*}
	e(t) =   \vvt(t,t)-u(t),
\end{equation*}
where $u$ is the (smooth) solution of \eqref{eq:eq}.

Since $\vvt$ is an extension of $u^n$ to all of $[0,T]$, it is clear that 
Theorem \ref{thm:godunov} is an immediate  consequence of 
the following lemma, which is our main result in this subsection.
\begin{lemma}\label{lem:godunov}
Let $T>0$ be given and assume $u_0 \in H^k$ with $6 \leq k \in \mathbb{N}$.
Then, for $\Delta t$ sufficiently small we have the following:
\begin{enumerate}
	\item  The Godunov method in Definition \ref{def:godunov}
	is well-posed with $\vvt \in C(\Om_{\Delta t};H^k)$. 
	\item The error $e(t)$ satisfies
	$$
	\|e(t)\|_{H^{k-\max\{\alpha,2\}}} = \|\vvt(t,t) - u(t)\|_{H^{k-\max\{\alpha,2\}}} 
	\leq  C \|u_0\|_{H^k}^2 t\, \Delta t.
	$$
\end{enumerate}
\end{lemma}
The proof of Lemma \ref{lem:godunov} will be a consequence of the results 
stated and proved in the ensuing subsections. 
The closing arguments will be given in Section \ref{sec:godunov}.

\subsection{Error evolution equations}

The main benefit of the new definition of the Godunov method 
is that it allows us to derive continuous-in-time evolution equations 
for the error $e$. Though these equations was 
previously derived in \cite{Holden:tao}, we repeat the
derivation here for the convenience of the reader.
For this purpose, 
we need the following Taylor expansion that holds for  
any smooth operator $E$
\begin{align*}
	E(f+g) &= E(f) + dE(f)[g] + \int_0^1 (1-\gamma)d^2E(f+\gamma g)[g]^2~d\gamma.
\end{align*}
Using the definition of $\vvt$ and the above Taylor formula, we deduce
\begin{equation}\label{eq:err1}
	\begin{split}
		e_t - dC(u)[e] 
		&= \vvt_t + \vvt_\tau - u_t - dA(u)[e] - dB(u)[e] \\
		&= \vvt_t + A(\vvt) - (A+B)(u) - dA(u)[e] - dB(u)[e] \\
		&= \vvt_t - B(\vvt) + (A(\vvt) - A(u) - dA(u)[e]) \\
		&\qquad \qquad +  (B(\vvt) - B(u) - dB(u)[e]) \\
		&= F(t,t) + \int_0^1 (1-\gamma)d^2C(u+ \gamma e)[e]^2~d\gamma,
	\end{split}
\end{equation}
where we have introduced the ``forcing" term
\begin{equation*}
	F(t, \tau) = \vvt_t(t, \tau)  - B(\vvt(t,\tau)).
\end{equation*}
By direct calculation,
\begin{equation}\label{eq:err3}
	\begin{split}
		F_\tau - dA(\vvt)[F] &= v_{t\tau} - B(\vvt)_\tau - dA(\vvt)[\vvt_t - B(\vvt)] \\
		&= A(\vvt)_t - dB(\vvt)[\vvt_\tau] - dA(\vvt)[\vvt_t] + dA(\vvt)[B(\vvt)] \\
		&= dA(\vvt)[\vvt_t] - dB(\vvt)[A(\vvt)] - dA(\vvt)[\vvt_t] + dA(\vvt)[B(\vvt)] \\
		&= [A,B](\vvt),
	\end{split}
\end{equation}
where we have  defined the commutator
\begin{equation*}
	\begin{split}
		[A,B](f) = dA(f)[B(f)] - dB(f)[A(f)].
	\end{split}
\end{equation*}

In our case, the operator $A$ is linear and hence
\begin{equation}\label{eq:aa}
	dA(f)[g] = A(g), \qquad d^2A(f)[g,h] = 0,
\end{equation}
while the operator $B$ satisfies
\begin{equation}\label{eq:bb}
	\begin{split}
		B(f) &= -\Div (f \vc{v}(f)), \\
		dB(f)[g] &= -\Div (f \vc{v}(g) + g \vc{v}(f)), \\
		d^2B(f)[g,h] &= -\Div (h \vc{v}(g) + g \vc{v}(h)).
	\end{split}
\end{equation}
This and the requirement (1) in Definition \ref{def:operators} yields
\begin{equation*}
	\begin{split}
		[A,B](f) 
		&= -A\left(\Div (f \vc{v}(f)) \right) +\Div \left(f\vc{v}(A(f))+A(f)\vc{v}(f)\right) \\
		&= -\Div \left(A(f \vc{v}(f)) - fA(\vc{v}(f))- A(f)\vc{v}(f)\right),
	\end{split}
\end{equation*}
which is exactly the divergence of the commutator 
appearing in (5) of Definition \ref{def:operators}.
Thus, the following lemma follows directly from this requirement.

\begin{lemma}\label{lem:comutator}
Let $f \in H^{k}$, where $\alpha$ is given 
by (3) in Definition \ref{def:operators}. Then,
$$
\|[A,B](f)\|_{H^{k-\max\{2, \alpha\}}} \leq C\|f\|_{H^{k}}^2.
$$
\end{lemma}
Using \eqref{eq:aa} and \eqref{eq:bb} in \eqref{eq:err1} and \eqref{eq:err3}, 
we obtain the following evolution equations for the error $e$
\begin{align}
	e_t + \Div \left(e \vc{v}(e) + u\vc{v}(e) +e\vc{v}(u) \right) - A(e)&= F, 
	\quad t \in (0,T)\label{eq:eeq}\\
	F_\tau - A(F)  &= [A,B](\vvt), \quad 
	(t,\tau) \in \Om_{\Delta t} \label{eq:feq}.
\end{align}
Upon inspection of these equations, we see that the error $e$ satisfies 
an equation similar to \eqref{eq:eq}, but
with an additional source term $F$.

\subsection{Estimates valid under the assumption of regularity}
In this subsection we state and prove some 
 results that  will be needed in order to prove Lemma \ref{lem:godunov}.
Due to the special configuration of the time domain, it will
be convenient to use the following notation 
for all times prior to a given time $(\sigma, \zeta) \in \Om_{\Delta t}$:
\begin{equation*}
	\Om_{\Delta t}^{\sigma, \zeta} 
	= \left\{(t,\tau) \in \Om_{\Delta t} \mid ~0\leq 
	t\leq \sigma,~0\leq \tau \leq \zeta\right\},
\end{equation*}

The following lemma is the most essential ingredient in the proof 
of Lemma \ref{lem:godunov}.

\begin{lemma}\label{lem:higher2}
	Let $\vvt(0,0) := u_0 \in H^{k}$.
	Assume the existence of a time $(\sigma, \zeta) \in [0,T]^2$ and 
	a finite constant $\gamma > 0$ such that
	$$
	\|\vvt(t,\tau)\|_{H^{k-\max\{\alpha,2\}}} 
	\leq \gamma, \quad  (t,\tau) \in \Om_{\Delta t}^{\sigma, \zeta}.
	$$
	There
	is a constant $C(\gamma)$  determined by $\gamma$, $u_0$, and $T$ such that
	$$
	\|\vvt(t,\tau)\|_{H^{k}} \leq C(\gamma), 
	\quad (t,\tau) \in \Om_{\Delta t}^{\sigma, \zeta}.
	$$
\end{lemma}
\begin{proof}
Fix any $(t,\tau) \in \Om_{\Delta t}^{\sigma,\tau}$ and let $n$ be such that
$ t,\tau\in [t_n, t_{n+1}]$.
Using the definition of $\vvt$ (Definition \ref{def:godunov}), we see that
\begin{equation*}
	\begin{split}
		\partial_\tau \frac{1}{2}\|\vvt(t,\tau)\|_{H^\ell}^2 
		&= \sum_{s= 0}^\ell \int_{\R^N} \Grad^s \vvt_\tau: \Grad^s \vvt~dx 
		= \sum_{s= 0}^\ell \int_{\R^N} A(\Grad^s \vvt): \Grad^s \vvt~dx \leq 0, 
	\end{split}
\end{equation*}
for any $\ell = 1, \ldots, k$,
where the last inequality is requirement (4) in Definition \ref{def:operators}.
The previous inequality yields
\begin{equation}\label{eq:high1}
	\|\vvt(t,\tau)\|_{H^k} \leq \|\vvt(t,t_n)\|_{H^k}.
\end{equation}
Next, we let $\ell \geq 4$ be an integer and 
apply Definition \ref{def:godunov} to obtain
\begin{equation*}
	\begin{split}
		\partial_t \frac{1}{2}\|\vvt(t,t_n)\|_{H^{\ell}}^2 
		&= -\sum_{s=0}^\ell\int_{\R^N}\Grad^s\Div (\vvt \vc{v}(\vvt)):\Grad^s \vvt~dx \\
		&\leq C\|\vvt\|_{H^{\ell-2}}\|\vvt\|_{H^{\ell}}^2,
	\end{split}
\end{equation*}
where we have applied Lemma \ref{lem:nonlineardiv}. If $\ell$ is such that 
$\ell -2 \leq k-\max\{2,\alpha\}$, we can apply 
the Gronwall lemma to the previous inequality to obtain
\begin{equation}\label{eq:high2}
	\|\vvt(t,t_n)\|_{H^{\ell}} \leq \|\vvt(t_n,t_n)\|_{H^{\ell}} e^{C\gamma \Delta t}.
\end{equation}
By combining \eqref{eq:high1} and \eqref{eq:high2}, we thus obtain the bound
\begin{equation}\label{eq:high3}
	\|\vvt(t, \tau)\|_{H^\ell} \leq \|u_0\|_{H^\ell}e^{C\gamma T}, 
	\quad 4 \leq \ell \leq k-\max\{2,\alpha\}.
\end{equation}

Now, we can repeat the above arguments with a 
new $\gamma := C(\gamma)= \|u_0\|_{H^\ell}e^{C\gamma T}$, 
to obtain
\begin{equation*}
	\|\vvt(t, \tau)\|_{H^{\ell+2}} \leq \|u_0\|_{H^{\ell+2}}e^{C(\gamma) T}, 
	\quad 4 \leq \ell \leq k-\max\{2,\alpha\}.
\end{equation*}
Clearly, we can repeat this process $n$ times until $\ell + 2n = k$ ($\alpha$ even)
or until $\ell + 2n = k-1$ ($\alpha$ odd). In the last case, we can conclude that
\begin{equation*}
	\|\vvt(t, \tau)\|_{H^{k-2}} \leq \|\vvt(t, \tau)\|_{H^{k-1}} 
	\leq \|u_0\|_{H^{k-1}}e^{C(\gamma) T},
\end{equation*}
and hence the arguments leading to \eqref{eq:high3} hold for $l=k$.
\end{proof}

In the following lemma, we give our main error estimate. Note 
that the result is only valid under a strong regularity assumption 
on the splitting solution $\vvt$. This assumption will be fully justified 
when we prove Lemma \ref{lem:godunov} in the next subsection.

\begin{lemma}\label{lem:godunoverror}
	Assume the existence of  $(\sigma, \zeta) \in [0,T]^2$ such 
	that the splitting solution 
	$$
	\|\vvt(t,\tau)\|_{H^k} \leq C(\gamma), 
	\quad (t,\tau) \in \Om_{\Delta t}^{\sigma, \zeta}.
	$$
	Then,
	\begin{equation*}
		\|e(t)\|_{H^{k-\max\{\alpha,2\}}} 
		\leq  \tilde{C}(\gamma) \Delta t, 
		\quad t \leq \sigma.
	\end{equation*}
\end{lemma}

\begin{proof}
To shorten the notation in this proof we  write $\ell=k-\max\{\alpha,2\}$.
Let us commence by estimating the size of the source term $F$. 
For this purpose, we apply $\Grad^s$ 
to \eqref{eq:feq}, multiply the result with $\Grad^s F$, integrate
by parts, and sum over $s=0, \ldots, k$, to obtain 
\begin{equation*}
	\begin{split}
		\frac{1}{2}\partial_\tau\|F(t,\tau)\|_{H^{\ell}}^2 
		&=\sum_{s=0}^{\ell} \int_{\R^N}\big( A(\Grad^s F):\Grad^s F
		+  \Grad^s [A, B](\vvt):\Grad^sF\big)~dx \\
		&\leq \left|\sum_{s=0}^{\ell}\int_{\R^N} \Grad^s [A, B](\vvt):\Grad^sF~dx\right| 
		\leq C\|F\|_{H^{\ell}}\left\|[A, B](\vvt)\right\|_{H^{\ell}}.
	\end{split}
\end{equation*}
Since $\ell = k-\max\{\alpha,2\}$, we can apply Lemma \ref{lem:comutator} to obtain
\begin{equation*}
	\partial_\tau \|F(t,\tau)\|_{H^{\ell}} \leq \left\|[A, B](\vvt)\right\|_{H^{\ell}} 
	\leq C\|\vvt\|_{H^{k}}^2 \leq C(\gamma).
\end{equation*}
By definition $F(t, t_n) = 0$. Hence, integration in the $\tau$ direction 
from $t_n$ to $\tau$ yields
\begin{equation}\label{eq:Fbound}
	\|F(t,\tau)\|_{H^{\ell}} \leq C(\gamma)\Delta t.
\end{equation}
	
Let us now turn to the evolution equation \eqref{eq:eeq} for the error.
By applying $\Grad^s$ to \eqref{eq:eeq}, multiplying the result with $\Grad^s e$, integrating
by parts, and summing over $s=0, \ldots, k$, we obtain	
\begin{equation*}
	\begin{split}
		\frac{1}{2}\partial_t \|e(t)\|_{H^{\ell}}^2 
		&= \sum_{s=0}^\ell\Big(\int_{\R^N} A(\Grad^s e):\Grad^s e~dx 
		+ \int_{\R^N}\Grad^s F:\Grad^se~dx\Big) \\
		&\qquad - \sum_{s=0}^\ell\int_{\R^N} \Grad^s\left(\Div \left(e \vc{v}(e) 
		+ u\vc{v}(e) +e\vc{v}(u)\right)\right):\Grad^s e~dx \\
		&\leq C\|F(t,t)\|_{H^{\ell}}\|e(t)\|_{H^{\ell}} 
		+ C\|u(t)\|_{H^{\ell+1}}\|e(t)\|_{H^\ell}^2, 
	\end{split}
\end{equation*}
where the last inequality is an application of Lemma \ref{lem:lineardiv} and 
 H\"older's inequality. Now, since $u_0 \in  H^{k}$, requirement 
(6) in Definition \ref{def:operators} tells us that the analytical solution $u \in C([0,T]; H^{k})$.
Thus, the previous inequality leads us to the conclusion
\begin{equation*}
	\partial_t\|e(t)\|_{H^\ell} \leq C\|F(t,t)\|_{H^\ell}
	+ C\|e(t)\|_{H^\ell} \leq C\left(\Delta t + \|e(t)\|_{H^\ell}\right),
\end{equation*}
where the last inequality is \eqref{eq:Fbound}. 
An application of the Gronwall inequality concludes the proof.
\end{proof}

\subsection{Proof of Lemma \ref{lem:godunov}}\label{sec:godunov}
We will determine the maximal size of $\Delta t > 0$ during 
the course of the proof. Specifically, the size of $\Delta t$ will 
be determined in accordance to the maximal existence time 
of solutions to the nonlinear transport equation \eqref{eq:transporteq}.
For the convenience of the reader we recall that the maximal 
existence time $T^*$ of solutions $u$ to \eqref{eq:transporteq} is 
characterized by
\begin{equation}\label{eq:blowup}
	\lim_{t \rightarrow T^*}\|u\|_{H^k} = +\infty.
\end{equation}

1. The proof of well-posedness ((1)~in Lemma \ref{lem:godunov}) will 
be a direct consequence of the following lemma. 
\begin{lemma}\label{lem:thepoint}
	Let $k\geq 6$ and $u_0 \in H^k$.
	There exists constants $\beta$, $\gamma$, depending 
	only on $u_0$ and $T$, such that if $\Delta t < \beta$ and
	\begin{equation*}\label{eq:c1}
		\|\vvt(t,\tau)\|_{H^{k-\max\{\alpha,2\}}} \leq 
		\gamma, \quad  (t,\tau) \in \Om_{\Delta t}^{\sigma, \zeta},
	\end{equation*}
	for some $(\sigma, \zeta) \in \Om_{\Delta t}$, then 
	\begin{equation*}\label{eq:c2}
		\|\vvt(t,\tau)\|_{H^{k-\max\{\alpha,2\}}} 
		\leq \frac{\gamma}{2}, \quad  (t,\tau) \in \Om_{\Delta t}^{\sigma, \zeta}.
	\end{equation*}	
\end{lemma}

Let us for the moment take Lemma \ref{lem:thepoint} for granted and 
explain why it concludes our proof of Lemma \ref{lem:godunov}.
For this purpose, we let $\Delta t < \beta$, where 
$\beta$ is dictated by Lemma \ref{lem:thepoint}.
Since the initial data $\vvt(0,0)=u_0 \in H^k$, our short-time 
existence result (Lemma \ref{lem:local}) enables us to start constructing
$\vvt(t,0)$ according to Definition \ref{def:godunov} up 
to some time $T^*$, that is, to start applying 
the $B$ operator (the nonlinear transport equation \eqref{eq:transporteq}) in
the first time step. Now, on the interval $(0,T^*)$
there is no problem determining $\gamma$ in accordance 
to both Lemma \ref{lem:thepoint} and such that
$$
\|\vvt(t,0)\|_{H^{k-\max\{\alpha,2\}}} \leq \gamma, \quad 0 \leq t < T^*.
$$
Lemma \ref{lem:thepoint} can then be applied to conclude that 
\begin{equation*}
	\|\vvt(t,0)\|_{H^{k-\max\{\alpha,2\}}} 
	\leq \frac{\gamma}{2}, \quad 0 \leq t < T^*.	
\end{equation*}
However, in view of \eqref{eq:blowup}, and since the 
norm is continuous in time there must exist an $\epsilon > 0$ 
such that 
\begin{equation*}
	\|\vvt(t,0)\|_{H^{k-\max\{\alpha,2\}}} \leq \gamma, 
	\quad 0 \leq t \leq T^*+ \epsilon.
\end{equation*}
Thus, by repeating these three steps, we can conclude that
$$
\|\vvt(t,0)\|_{H^{k-\max\{\alpha,2\}}} 
\leq \gamma, \quad 0 \leq t \leq \Delta t.
$$
We can now start to determine $\vvt(t,\tau)$, for $0< \tau\leq \Delta t$, 
according to Definition \ref{def:godunov}, that is, to start solving with the $A$ 
operator (the linear equation \eqref{eq:lineareq}).
By virtue of our existence result 
for \eqref{eq:lineareq} (Lemma \ref{lem:global}), we can 
repeat the above argument to conclude that $\vvt(t,\tau)$ satisfies
\begin{equation*}
	\|\vvt(t,\tau)\|_{H^{k-\max\{\alpha,2\}}} 
	\leq \gamma, \quad 0 \leq t,\tau \leq \Delta t.
\end{equation*}
Besides well-posedness of the first time step, this result tell 
us that $\gamma$ also bounds the $H^{k-\max\{\alpha,2\}}$ 
norm of $\vvt(\Delta t, \Delta t)$, which is indeed the initial data in
the next time step. Hence, we can repeat the entire process for all 
time steps recursively to conclude that
\begin{equation*}
	\|\vvt(t,\tau)\|_{H^{k-\max\{\alpha,2\}}} 
	\leq \gamma, \quad  (t,\tau) \in \Om_{\Delta t}.
\end{equation*}
Lemma \ref{lem:higher2} can the be applied to obtain
\begin{equation*}
	\|\vvt(t,\tau)\|_{H^k} \leq C(\gamma), 
	\quad  (t,\tau) \in \Om_{\Delta t},	
\end{equation*}
which concludes the proof of 
well-posedness, i.e., (1) of Lemma \ref{lem:godunov}.

2. Since we now know that $\|\vvt(t,\tau)\|_{H^k} \leq C(\gamma)$, $(t,\tau) \in \Om_{\Delta t}$,
the postulates of Lemma \ref{lem:godunoverror} are satisfied. Thus,
\begin{equation*}
	\|e(t)\|_{H^{k-\max\{\alpha,2\}}} \leq \Delta t C, \quad t \in (0,T),
\end{equation*}
which is precisely (2) in Lemma \ref{lem:godunov}. 

We conclude this section by proving Lemma \ref{lem:thepoint}.
\begin{proof}[Proof of Lemma \ref{lem:thepoint}]
To shorten the notation, we let $\ell = k-\max\{\alpha,2\}$. 
Assume that $(\sigma, \zeta) \in \Om_{\Delta t}$ is such that
\begin{equation}\label{eq:lm36}
	\|\vvt(t,\tau)\|_{H^{\ell}} \leq \gamma, \quad  (t,\tau) \in \Om_{\Delta t}^{\sigma, \tau},
\end{equation}
where the values of $\gamma$ and  $\Delta t$ are still to be determined.

From Lemma \ref{lem:higher2} we know that \eqref{eq:lm36} implies that
\begin{equation*}
	\|\vvt(t,\tau)\|_{H^{k}} \leq C(\gamma), 
	\quad  (t,\tau) \in \Om_{\Delta t}^{\sigma, \tau}.
\end{equation*}
We can also apply Lemma \ref{lem:godunoverror} to obtain
\begin{equation}\label{eq:inv}
	\|e(t)\|_{H^\ell} \leq \Delta t \tilde C(\gamma).
\end{equation}

Now, fix any $(t,\tau) \in \Om_{\Delta t}^{\sigma, \tau}$. 
By definition of the Godunov method (Definition~\ref{def:godunov}),
\begin{equation*}
	\begin{split}
		\left|\partial_\tau\frac{1}{2}\|\vvt(t,\tau)\|_{H^{\ell}}^2 \right|
		&= \left|\sum_{s=0}^k\int_{\R^N}\Grad^s \vvt_\tau:\Grad^s \vvt~dx\right| 
		=\left|\sum_{s=0}^k\int_{\R^N}A(\Grad^s\vvt):\Grad^s \vvt~dx\right| \\
		&\leq C\|\vvt\|_{H^{\ell}}\|A(\vvt)\|_{H^{\ell}} 
		\leq C\|\vvt\|_{H^{\ell}}\|\vvt\|_{H^{\ell+ \alpha}}\leq C\|\vvt\|_{H^{\ell}}\|\vvt\|_{H^{k}},
	\end{split}
\end{equation*}
from which it follows that
\begin{equation*}
	\left|\partial_\tau\|\vvt(t,\tau)\|_{H^{\ell}}\right| \leq C\|\vvt\|_{H^{k}}.
\end{equation*}
Using this inequality and  Lemma \ref{lem:higher2}, we find
\begin{equation*}
	\begin{split}
		\|\vvt(t,\tau)\|_{H^\ell} 
		&\leq \|\vvt(t,t)\|_{H^\ell} 
		+ \int^\tau_t\left|\partial_\tau \|\vvt(t,\tilde\tau)\|_{H^k}\right|~d|tilde\tau \\
		&\leq \|\vvt(t,t)\|_{H^\ell} + \Delta tC\sup_{\tilde\tau \in [t,\tau]}\|\vvt(t,\tilde\tau)\|_{H^k} \\
		&\leq \|\vvt(t,t)\|_{H^\ell} + \Delta t C(\gamma)\\
			&\leq \|e(t)\|_{H^\ell} + \|u(t)\|_{H^\ell} + \Delta tC(\gamma) \\
		&\leq \Delta t\tilde C(\gamma) + C_2,
	\end{split}
\end{equation*}
by applying \eqref{eq:inv},
we obtain the estimate
\begin{equation}\label{eq:lst}
	\begin{split}
		\|\vvt(t,\tau)\|_{H^\ell} 
		&\leq \|e(t)\|_{H^\ell} + \|u(t)\|_{H^\ell} + \Delta tC(\gamma) \\
		&\leq \Delta tC(\gamma) + C_2.
	\end{split}
\end{equation}
Finally, by fixing $\gamma$ and $\Delta t$ according to
\begin{equation*}
	\gamma = 4C_2, \qquad \Delta t \leq \frac{C_2}{\tilde C(\gamma)}:=\beta,
\end{equation*}
and applying this information to \eqref{eq:lst}, we see that
\begin{equation*}
	\|\vvt(t,\tau)\|_{H^\ell} \leq \frac{\gamma}{2},
\end{equation*}
which concludes the proof.
\end{proof}

\section{Strang splitting (proof of Theorem \ref{thm:strang})}\label{sec:strang-split}
As for the Godunov method, our convergence analysis will require a continuous 
definition of the Strang method. In contrast to the Godunov case, 
we will now introduce three time variables instead of two \cite{Holden:karper}.
We consider the domain:
\begin{equation*}
	\Om_{\Delta t} =  \bigcup_{n=0}^{\lfloor T/\Delta t\rfloor-1}[\frac{t_n}{2}, 
	\frac{t_{n+1}}{2}] \times [t_n, t_{n+1}] \times [\frac{t_n}{2}, \frac{t_{n+1}}{2}].
\end{equation*}
The continuous Strang method is given by the following definition.
\begin{definition}\label{def:strang}
	For $\Delta t > 0$ given, we say that $\vvt$ is the Strang 
	splitting approximation to \eqref{eq:eq}
	whenever $\vvt$ solves
	\begin{equation}\label{eq:strang-2}
		\begin{split}
			\vvt(0,0,0) &= u_0, \\
			\vvt_t(t,t_n,\frac{t_n}{2}) &= B(\vvt(t,t_n,\frac{t_n}{2})), 
			\quad t \in (\frac{t_n}{2}, \frac{t_{n+1}}{2}], \\
			\vvt_\tau(t, \tau,\frac{t_n}{2}) &=  
			A(\vvt(t, \tau,\frac{t_n}{2})), \quad 
			(t,\tau) \in [\frac{t_n}{2}, \frac{t_{n+1}}{2}]\times (t_n,t_{n+1}], \\
			\vvt_\omega(t, \tau, \omega)&= B(\vvt(t,\tau, \omega)), 
			\quad (t, \tau, \omega) \in [\frac{t_n}{2}, \frac{t_{n+1}}{2}]
			\times [t_n,t_{n+1}] \times (\frac{t_n}{2}, \frac{t_{n+1}}{2}).
		\end{split}
	\end{equation}
\end{definition}

In each box $[\frac{t_n}{2}, \frac{t_{n+1}}{2}] 
\times [t_n, t_{n+1}] \times [\frac{t_n}{2}, \frac{t_{n+1}}{2}]$, we will 
mainly consider the function $\vvt$ along the diagonal, i.e., the
function $\vvt(\frac{t}{2},t,\frac{t}{2})$ for $t \in [t_n, t_{n+1}]$. 
Observe that 
each point on this diagonal is a Strang splitting solution for a specific time step.
More precisely, $\vvt(\frac{t}{2},t,\frac{t}{2})$ is a Strang splitting solution with time step 
$t - t_n$. An easy consequence is that
\begin{equation*}
	\vvt (\frac{t_n}{2}, t_n,\frac{t_n}{2}) = u^n, \quad n=0, \ldots, \lfloor T/\Delta t\rfloor,
\end{equation*}
and hence that $\vvt(\frac{t}{2},t,\frac{t}{2})$ can be seen as an extension of $\{u^n\}_n$
to all of $[0,T]$.

To measure the error, we will use the function
\begin{equation*}
	e(t) =  \vvt(\frac{t}{2},t,\frac{t}{2})-u(t),
\end{equation*}
where $u$ is the (smooth) solution of \eqref{eq:eq}.

Theorem \ref{thm:strang} is an immediate 
consequence of the two following lemmas.
\begin{lemma}[Well-posedness]\label{lem:strang}
Let $\vvt$ the Strang splitting solution of \eqref{eq:eq}
in the sense of  Definition \ref{def:strang} and \eqref{eq:strang-2}.
Let $T>0$ be given and assume $u_0 \in H^k$ with $6 \leq k \in \mathbb{N}$.
Then, for $\Delta t>0$ sufficiently small
\begin{enumerate}
	\item The Strang method  
	is well-posed with $\vvt \in C(\Om_{\Delta t};H^k)$. 
	\item The error $e(t)$ satisfies
	$$
	\|e(t)\|_{H^{k-\max\{\alpha,2\}}} = 
	\|\vvt(\frac{t}{2},t,\frac{t}{2}) - u(t)\|_{H^{k-\max\{\alpha,2\}}} \leq tC\Delta t.
	$$
\end{enumerate}
\end{lemma}

\begin{lemma}[Convergence]\label{lem:rate}
Let $\vvt$ the Strang splitting solution of \eqref{eq:eq}
in the sense of  Definition \ref{def:strang} and \eqref{eq:strang-2}.
Let $T>0$ be given and assume $u_0 \in H^k$ with $6 \leq k \in \mathbb{N}$.
Then, for $\Delta t>0$ sufficiently small
\begin{equation*}
	\|e(t)\|_{H^{k-3\max\{\alpha,1\} }} \leq C(\Delta t)^2.
\end{equation*}
\end{lemma}

Lemmas \ref{lem:strang} and \ref{lem:rate} will 
be consequences of the results stated and proved in 
the ensuing subsections. 

\subsection{Error evolution equations}
By direct calculation, we see that the error $e$ satisfies 
the time-evolution
\begin{equation*}
	\begin{split}
		e_t - dC(u)[e] 
		&= \frac{\vvt_t}{2} + \vvt_\tau  + \frac{\vvt_\omega}{2} -u_t - dA(u)[e] - dB(u)[e] \\
		&= \frac{\vvt_t}{2} + \vvt_\tau + \frac{1}{2}B(\vvt) - (A+B)(u) - dA(u)[e] - dB(u)[e] \\
		&= \frac{1}{2}\left(\vvt_t - B(\vvt)\right) + \vvt_\tau - A(\vvt) \\
		&\qquad + (A(\vvt) - A(u) - dA(u)[e]) 
		 + (B(\vvt) - B(u) - dB(u)[e])  \\
		&= F(t) + \int_0^1 (1-\gamma)d^2 
		C(u + \gamma e)[e]^2~d\gamma, 
	\end{split}
\end{equation*}
where $F(t) = F(\frac{t}{2}, t, \frac{t}{2})$ and
\begin{equation*}
	F(t,\tau,\omega) = 
	\frac{1}{2}\left(\vvt_t(t,\tau,\omega) 
	- B(\vvt(t,\tau, \omega))\right)
	+\vvt_\tau - A(\vvt(t,\tau,\omega)).
\end{equation*}

Since $\vvt_\tau - A(\vvt(t,\tau,\omega)) = 0$ when $\omega =
\frac{t_n}{2}$, $n=0, \ldots,\lfloor T/\Delta t\rfloor-1$, we 
can apply the arguments of \eqref{eq:err3} to obtain 
\begin{equation*}
	F_\tau - dA(\vvt)[F] = \frac{1}{2}[A,B](\vvt), 
	\quad (t,\sigma,\frac{t_n}{2}) \in \Om_{\Delta t}
\end{equation*}
where as before
\begin{equation*}
	[A,B](f) = dA(f)[B(f)] - dB(f)[A(f)].
\end{equation*}
We also derive the following equation for 
the evolution of $F$ in $\omega$:
\begin{equation*}
	\begin{split}
		&F_\omega  - dB(\vvt)[F] \\
		&\qquad= \frac{1}{2}\vvt_{t\omega} - \frac{1}{2}B(\vvt)_\omega 
		- \frac{1}{2}dB(\vvt)[\vvt_t - B(\vvt)] \\
		&\qquad \qquad + \vvt_{\tau \omega} - dA(\vvt)[\vvt_\omega] 
		- dB(\vvt)[\vvt_\tau - A(\vvt)]\\
		&\qquad= \frac{1}{2}B(\vvt)_t - \frac{1}{2}dB(\vvt)[\vvt_\omega] 
		- \frac{1}{2}dB(\vvt)[\vvt_t] +  \frac{1}{2}dB(\vvt)[B(\vvt)] \\
		&\qquad \qquad + dB(\vvt)[\vvt_\tau] - dA(\vvt)[B(\vvt)] - dB(\vvt)[\vvt_\tau] 
		+ dB(\vvt)[A(\vvt)] \\
		&\qquad= \frac{1}{2}\left(dB(\vvt)[\vvt_t] - dB(\vvt)[B(\vvt)]- dB(\vvt)[\vvt_t] 
		+  dB(\vvt)[B(\vvt)]\right) \\
		&\qquad \qquad + dB(\vvt)[A(\vvt)] - dA(\vvt)[B(\vvt)].
	\end{split}
\end{equation*}
Thus, recalling the definition of $[\dott,\dott]$, we find 
\begin{equation*}
	F_\omega  - dB(\vvt)[F] = [B,A](\vvt).
\end{equation*}
For the class of equations \eqref{eq:eq}, the 
evolution equations read:
\begin{align}
	e_t + \Div \left(e \vc{v}(e) + u\vc{v}(e) +e\vc{v}(u) \right) - A(e)&= F, 
	\quad t \in (0,T)\label{eq:eeqS}\\
	F_\tau - A(F)  &= \frac{1}{2}[A,B](\vvt), 
	\quad (t,\tau, \frac{t_n}{2}) \in \Om_{\Delta t} \label{eq:feqS}, \\
	F_\omega + \Div\left(\vvt \vc{v}(F) + F\vc{v}(\vvt)\right)  & 
	= -[A,B](\vvt), \quad (t,\tau, \omega) \in \Om_{\Delta t} \label{eq:feq2S}.
\end{align} 

As for the Godunov method, it will be convenient to define a notation 
for all times prior to a given time $( \sigma, \zeta,\nu) \in \Om_{\Delta t}$:
\begin{equation*}
	\Om_{\Delta t}^{\sigma, \tau,\nu} = \left\{(t,\tau, \omega) 
	\in \Om_{\Delta t}\mid 0\leq t\leq \sigma,~0 
	\leq \tau \leq \zeta, ~0 \leq \omega\leq \nu\right\}.
\end{equation*}

To prove Lemma \ref{lem:strang}, we will need Strang versions
of Lemmas \ref{lem:higher2} and \ref{lem:godunoverror}.
Clearly, there is no problem to extend Lemma \ref{lem:higher2}
to conclude the following result.

\begin{lemma}\label{lem:higher-strang}
	Let $\vvt(0,0) := u_0 \in H^{k}$.
	Assume the existence of a time 
	$(\sigma, \tau,\nu) \in [0,T]^3$ and a 
	finite constant $\gamma > 0$ such that
	$$
	\|\vvt(t,\tau, \omega)\|_{H^{k-\max\{\alpha,2\}}} \leq \gamma, 
	\quad  (t,\tau, \omega) \in \Om_{\Delta t}^{\sigma, \zeta, \nu}.
	$$
	Then there is a constant $C(\gamma)$ 
	determined by $\gamma$, $u_0$, and $T$ such that
	$$
	\|\vvt(t,\tau, \omega)\|_{H^{k}} \leq C(\gamma), \quad  (t,\tau, \omega) 
	\in \Om_{\Delta t}^{\sigma, \zeta, \nu}.
	$$
\end{lemma}

Let us now prove the Strang version of Lemma \ref{lem:godunoverror}.
\begin{lemma}\label{lem:strangerror}
	Assume the existence of  $(\sigma, \zeta, \nu) \in \Om_{\Delta t}$ such 
	that the splitting solution 
	$$
	\|\vvt(t,\tau, \omega)\|_{H^k} \leq C(\gamma), \quad  (t,\tau,\omega) 
	\in \Om_{\Delta t}^{\sigma, \zeta, \nu}.
	$$
	Then,
	\begin{equation*}
		\|e(t)\|_{H^{k-\max\{\alpha,2\}}} \leq \Delta t \tilde{C}(\gamma), 
		\quad t \leq \sigma.
	\end{equation*}
\end{lemma}

\begin{proof}
Let $\ell = k-\max\{\alpha,2\}$.
As in the proof of Lemma \ref{lem:godunoverror}, we first 
estimate the size of the source term $F$.
For this purpose, fix any $(t,\tau, \omega) \in 
\Om^{\sigma, \zeta, \nu}_{\Delta t}$ and 
let $n$ be such that $(t,\tau,\omega) \in \left[\frac{t_n}{2}, 
\frac{t_{n+1}}{2}\right] \times \left[t_{n}, t_{n+1}\right]
\times \left[\frac{t_n}{2}, \frac{t_{n+1}}{2}\right]$.

Now, in the plane given by $\omega = \frac{t_n}{2}$, we have that $F$ satisfies 
the equation \eqref{eq:feqS}. Since this equation is similar to \eqref{eq:feq}, 
we can repeat the arguments in the 
proof of Lemma \ref{lem:godunoverror} to conclude that
\begin{equation}\label{eq:SF}
	\|F(t,\sigma,\frac{t_n}{2})\|_{H^{\ell}} \leq \Delta t C(\gamma).
\end{equation}
Hence, to estimate $F$ at the point $(t,\tau,\omega)$, we can integrate from $(t,\tau,\omega)$
in the $\omega$ direction to the plane given by $\omega= \frac{t_n}{2}$ 
and apply \eqref{eq:SF}. To achieve this, we first apply $\Grad^s$ to \eqref{eq:feq2S},
multiply with $\Grad^s F$, integrate, and sum over $s=0, \ldots, k$, to obtain
\begin{equation*}
	\begin{split}
		\partial_\omega \frac{1}{2}\|F(t,\tau,\omega)\|_{H^{\ell}}^2
		&= -\sum_{s=0}^\ell\int_{\R^N}\Grad^s\Div \left(\vvt \vc{v}(F) 
		+ F\vc{v}(\vvt)\right):\Grad^s F~  dx \\
		&\qquad - \sum_{s=0}^\ell\int_{\R^N}\Grad^s[A,B](\vvt):\Grad^s F~dx \\
		&\leq C\left(\|\vvt\|_{H^{\ell+1}}\|F\|_{H^\ell}^2 
		+ \|\vvt\|_{H^\ell}\|F\|_{H^\ell}^2\right) + C\|\vvt\|_{H^k}^2\|F\|_{H^\ell}.
	\end{split}
\end{equation*}
In the last inequality we have applied Lemmas \ref{lem:lineardiv} and \ref{lem:comutator}.
Since $\vvt \in H^{k}$, we can integrate the last inequality to obtain
\begin{equation*}
	\|F(t,\tau,\omega)\|_{H^\ell} \leq \|F(t,\sigma,\frac{t_n}{2})\|_{H^{\ell}} 
	+ C(\gamma)\Delta t \leq \Delta tC(\gamma),
\end{equation*}
where the last inequality is \eqref{eq:SF}.

Since \eqref{eq:eeqS} is identical to \eqref{eq:eeq}, we can repeat the arguments in 
the proof of Lemma \ref{lem:godunoverror} to conclude the proof.
\end{proof}

\subsection{Proof of well-posedness (Lemma \ref{lem:strang})}
We can apply  
similar arguments as those of Section \ref{sec:godunov}
(for the  Godnov method) to prove Lemma \ref{lem:strang}. In particular, 
upon inspection of the arguments in Section \ref{sec:godunov}, it 
is clear that Lemma \ref{lem:strang} is a consequence of
the following result.

\begin{lemma}\label{lem:thepoint2}
	Let $k\geq 6$ and $\vvt(0,0,0):=u_0 \in H^k$.
	There exists constants $\beta$ and $\gamma$, depending only 
	on $T$ and $u_0$, such that if $\Delta t < \beta$ and
	$$
	\|\vvt(t,\tau, \omega)\|_{H^{k-\max\{\alpha,2\}}} \leq \gamma, 
	\quad  (t,\tau,\omega) \in \Om_{\Delta t}^{\sigma, \zeta, \nu},
	$$
	for some $(\sigma, \zeta, \nu) \in \Om_{\Delta t}$, then
	\begin{equation*}
	 	\|\vvt(t,\tau, \omega)\|_{H^{k-\max\{\alpha,2\}}} \leq \frac{\gamma}{2}, 
		\quad  (t,\tau,\omega) \in \Om_{\Delta t}^{\sigma, \zeta, \nu}.
	\end{equation*}
\end{lemma}

\begin{proof}
Let $\ell = k - \max\{\alpha,2\}$ and assume that $(t,\tau,\omega) 
\in \Om_{\Delta t}^{\sigma, \zeta, \nu}$
is such that 
\begin{equation*}
	\|\vvt(t,\tau, \omega)\|_{H^{\ell}} \leq \gamma, \quad  
	(t,\tau,\omega) \in \Om_{\Delta t}^{\sigma, \zeta, \nu},
\end{equation*}
where the values of $\gamma$ and $\Delta t$ are to be determined. 
Lemma \ref{lem:higher-strang} can then be applied and yields
\begin{equation*}
	\|\vvt(t,\tau,\omega)\|_{H^{k}} \leq C(\gamma),\quad \forall (t,\tau,\omega) 
	\in \Om_{\Delta t}^{\sigma, \zeta, o}.
\end{equation*}
Furthermore, we can apply Lemma \ref{lem:strangerror} to conclude the estimate
\begin{equation*}\label{eq:inv2}
	\|e(t)\|_{H^\ell} \leq \Delta t \tilde C(\gamma). 
\end{equation*}

Now, to proceed we fix any $(t, \tau, \omega) \in \Omega_{\Delta t}^{\sigma, \zeta, \nu}$.
Using the definition of the Strang method (Definition \ref{def:strang}) and 
Lemma \ref{lem:nonlineardiv}, we see that
\begin{equation*}
	\begin{split}
		\left|  \partial_\omega \frac{1}{2}\|\vvt(t,\tau,\omega)\|_{H^\ell}^2\right|
		& = \left|\sum_{s=0}^\ell \int_{\R^N}\Grad^s \vvt_\omega : \Grad^s\vvt~dx  \right| \\
		& =\left|-\sum_{s=0}^\ell \int_{\R^N}
		\Grad^s\Div \left(\vvt \vc{v}(\vvt) \right):\Grad^s \vvt ~dx  \right| \\
		&\leq	 C\|\vvt\|_{H^{\ell-2}} \|\vvt\|_{H^\ell}^2
		\leq C\|\vvt\|_{H^\ell}^3 \leq C\gamma^3. 
	\end{split}
\end{equation*}
From this it easily follows that
\begin{equation}\label{eq:Xtra}
	\|\vvt(t,\tau, \omega)\|_{H^\ell} \leq \|\vvt(t,\tau, t)\|_{H^\ell} + \Delta t C(\gamma).
\end{equation}
By the same calculations as those in the proof of Lemma \ref{lem:thepoint}, we also deduce
\begin{equation*}
	\begin{split}
		 \|\vvt(t,\tau, t)\|_{H^\ell} &\leq \|\vvt(t,2t, t)\|_{H^\ell} + \Delta t C(\gamma) \\
		&\leq \|e(2t)\|_{H^\ell} + \|u(2t)\|_{H^\ell} + \Delta t C(\gamma) \\
		&\leq C(\gamma) \Delta t + C_2.
	\end{split}
\end{equation*}
Using this in \eqref{eq:Xtra} allow us to conclude
\begin{equation}\label{eq:soon}
	\|\vvt(t,\tau, \omega)\|_{H^\ell} \leq C(\gamma) \Delta t + C_2.
\end{equation}
Since $(t, \tau, \omega)$ was  arbitrary, we can conclude that 
\eqref{eq:soon} holds for all $(t, \tau, \omega) \in \Om_{\Delta t}^{\sigma, \zeta, \nu}$.

Finally, we fix $\gamma$ and $\Delta t$ according to
\begin{equation*}
	\gamma = 4C_2, \qquad \Delta t\leq \frac{C_2}{C(\gamma)}:=\beta,
\end{equation*}
in \eqref{eq:soon} to obtain
\begin{equation*}
	\|\vvt(t,\tau, \omega)\|_{H^\ell} \leq \frac{\gamma}{2}, 
	\quad  (t, \tau, \omega) \in \Om_{\Delta t}^{\sigma, \zeta, \nu}.
\end{equation*}
\end{proof}

Equipped with the previous lemma, Lemma \ref{lem:strang} can be proved as 
we did with Lemma \ref{lem:godunov} in the Godunov case.

\subsection{Temporal regularity}
To prove second-order convergence (Lemma \ref{lem:rate}), we will 
need regularity in time of our splitting solution. 
Since we are working with three time variables, it will be convenient to
recall the notation
\begin{equation*}
  \nabla_t^l f= \frac{\partial^{\abs{l}}f}{\partial
    t^{l_1}\partial \tau^{l_2} \partial \omega^{l_3}}, \qquad \abs{l}=l_1+l_2+l_3,
\end{equation*}
for a multi-index $l=(l_1,l_2,l_3)$. We will also use the  notation
\begin{equation*}
  \nabla_t^k f= \left\{ \frac{\partial^{\abs{l}}f}{\partial
    t^{l_1}\partial \tau^{l_2} \partial \omega^{l_3}}\mid   \abs{l}=k\right\}
\end{equation*}
for any natural number $k$
\begin{lemma}\label{lem:tempv}
Let $\vvt$ be the Strang splitting solution in 
the sense of Definition \ref{def:strang}.
If $u_0 \in H^{k}$, then
\begin{equation*}
	\|\nabla^l_{ t} \vvt(t,\tau,\omega)\|_{H^{k-|l|\max\{\alpha,1\} }} \leq C\,,
\quad  (t,\tau,\omega)\in\Om_{\Delta t}.
\end{equation*}
\end{lemma}

\begin{proof}
We will argue by induction on $|l|$. For $|l|=0$, the result 
follows from Lemma \ref{lem:strang}. 
To proceed we assume that the result holds for $|l|=0, \ldots, q$. To close 
the induction argument it remains to prove the result for $|l|=q+1$.

To simplify notation, we let 
$$
\lambda= k-|l|\max\{\alpha,1\} = k-(q+1)\max\{\alpha,1\}.
$$
Now, let $(t',\tau',\omega') \in \Omega_{\Delta t}$ be arbitrary and fix $n$ such 
that $(t',\tau',\omega')
\in[\frac{t_n}{2},\frac{t_{n+1}}{2}]\times[t_n,t_{n+1}]\times[\frac{t_n}{2},\frac{t_{n+1}}{2}]$. An arbitrary 
component of $\nabla_{t}^{l} \vvt$ can be written in the form
\begin{equation*}
	\Theta^{q+1}_{i,j,\ell} = \frac{\partial^{q+1}}{\partial t^{i}\partial\tau^{j}\partial\omega^{\ell}}\vvt, \quad 0 \leq i,j,\ell \leq q+1, \quad i+j+\ell = q+1, \quad i,j,\ell \in \mathbb{N}.
\end{equation*}
To estimate 
the arbitrary component at the point $(t',\tau',\omega')$, we apply
the fundamental theorem of calculus to obtain
\begin{equation}\label{eq:fundthm}
	\begin{split}
		&\frac{1}{2}\|\Theta_{i,j,\ell}^{q+1}(t',\tau',\omega')\|_{H^\lambda}^2 \\
		&\quad
                =\frac{1}{2}\|\Theta_{i,j,\ell}^{q+1}(t',\tau',\frac{t_n}{2})\|_{H^\lambda}^2\\
                &\qquad + \sum_{s=0}^\lambda\int_{\frac{t_n}{2}}^{\omega'}
                \int_{\R^N}\Grad^s\partial_\omega
                \Theta_{i,j,\ell}^{q+1}(t',\tau',\tilde\omega):\Grad^s
                \Theta_{i,j,\ell}^{q+1}(t',\tau',\tilde\omega)~dxd\tilde\omega \\
                &\quad = \frac{1}{2}\|\Theta_{i,j,\ell}^{q+1}(t',\tau',\frac{t_n}{2})\|_{H^\lambda}^2\\
                &\quad \qquad + \sum_{s=0}^\lambda \sum_{r=0}^i \sum_{n=0}^j
                \sum_{m=0}^{\ell}\left(i \atop r\right)\left(j \atop n\right)\left(\ell \atop m\right)\\
                &\quad \qquad  \qquad \times \int_{\frac{t_n}{2}}^{\omega'}
                \int_{\R^N}\Grad^s\Div\left(\Theta_{r,n,m}^{m+n+r}(t',\tau',\tilde\omega)
                \vc{v}\left(\Theta_{i-r,j-n,\ell-m}^{\abs{l}-r-n-m}(t',\tau',\tilde\omega)\right)\right) \\
                &\quad \qquad \qquad \qquad 
                \qquad: \Grad^s\Theta_{i,j,\ell}^{q+1}(t',\tau',\tilde\omega)~dxd\tilde\omega,
	\end{split}
\end{equation}
where we used the definition of $\Theta_{i,j,\ell}^{q+1}$ 
and $\vvt_\omega$ (cf.~\eqref{eq:strang-2}) to conclude the last equality. 
Let us consider three separate cases of $m+n+r$ in the quadruple sum above.\\
(i) If $m+n+r = q+1$, the corresponding term in the above reads
\begin{align}\label{eq:indu1}
	&\sum_{s=0}^\lambda\int_{\frac{t_n}{2}}^{\omega'}\int_{\R^N}
	\Grad^s\Div\left(\Theta_{i,j,\ell}^{q+1}(t',\tau',\tilde\omega) 
	\vc{v}\left(\vvt (t',\tau',\tilde\omega)\right)\right): 
	\Grad^s\Theta_{i,j,\ell}^{q+1}(t',\tau',\tilde\omega)~dxd\tilde\omega \nonumber\\
	&\qquad \leq C\int_{\frac{t_n}{2}}^{\omega'}
	\left\|\vc{v}(\vvt (t',\tau',\tilde\omega))\right\|_{H^\lambda}
	\left\|\Theta_{i,j,\ell}^{q+1}(t',\tau',\tilde\omega)\right\|_{H^\lambda}^2~d\tilde\omega \nonumber\\
	&\qquad \leq C\int_{\frac{t_n}{2}}^{\omega'}
	\left\|\Theta_{i,j,\ell}^{q+1}(t',\tau',\tilde\omega)\right\|_{H^\lambda}^2~d\tilde\omega,
\end{align}
where we have applied Lemmas \ref{lem:lineardiv} and  \ref{lem:higher-strang}. \\
(ii) If $m+n+r=0$, we can also apply 
Lemmas \ref{lem:lineardiv} and \ref{lem:higher-strang} to conclude
\begin{align*}
	&\sum_{s=0}^\lambda\int_{\frac{t_n}{2}}^{\omega'}\int_{\R^N}
	\Grad^s \Div \left(\vvt (t',\tau',\tilde\omega)
	\vc{v}\left(\Theta_{i,j,\ell}^{q+1}(t',\tau',\tilde\omega)\right)\right): 
	\Grad^s\Theta_{i,j,\ell}^{q+1}(t',\tau',\tilde\omega)~dxd\tilde\omega \nonumber\\
	&\qquad \leq C\int_{\frac{t_n}{2}}^{\omega'}
	\|\vvt(t',\tau',\tilde\omega)\|_{H^{\lambda+1}}
	\|\Theta_{i,j,\ell}^{q+1}(t',\tau',\tilde\omega)\|_{H^\lambda}^2~d\tilde\omega \nonumber \\
	&\qquad \leq C\int_{\frac{t_n}{2}}^{\omega'}
	\|\Theta_{i,j,\ell}^{q+1}(t',\tau',\tilde\omega)\|_{H^\lambda}^2~d\tilde\omega.
\end{align*}
(iii) For the remaining cases ($1 \leq m+n+r \leq q$), we apply the H\"older inequality 
to obtain
\begin{align}\label{eq:indu3}
	&\sum_{s=0}^\lambda\int_{\frac{t_n}{2}}^{\omega'}
	\int_{\R^N}\Grad^s \Div \left(\Theta_{r,n,m}^{m+n+r}(t',\tau',\tilde\omega)
	\vc{v}\left(\Theta_{i-r,j-n,\ell-m}^{q+1-r-n-m}(t',\tau',\tilde\omega)\right)\right)\nonumber \\
	&\qquad \qquad\qquad\qquad
	: \Grad^s\Theta_{i,j,\ell}^{q+1}(t',\tau',\tilde\omega)~dxd\tilde\omega \nonumber \\
	&\leq C\int_{\frac{t_n}{2}}^{\omega'}\|\Theta_{i,j,\ell}^{q+1}(t',\tau',\tilde\omega)\|_{H^\lambda}
	\|\Theta_{r,n,m}^{m+n+r}(t',\tau',\tilde\omega)\|_{H^{\lambda+1}}\\
	&\qquad \qquad\qquad\qquad \times
	\|\Theta_{i-r,j-n,\ell-m}^{{q+1}-r-n-m}(t',\tau',\tilde\omega)\|_{H^{\lambda+1}}~d\tilde\omega
	\nonumber \\
	&\leq C \int_{\frac{t_n}{2}}^{\omega'}
        \|\Theta_{i,j,\ell}^{q+1}(t',\tau',\tilde\omega)\|_{H^\lambda}~d\tilde\omega,\nonumber
\end{align}
where we have used that our induction hypothesis yields
\begin{equation*}
	\|\nabla_t^{q} \vvt\|_{H^{\lambda + 1}}= \|\nabla_t^{q} \vvt\|_{H^{k + 1 - (q+1)\max\{\alpha,1\}}} 
	\leq \|\nabla_t^{q} \vvt\|_{H^{k - q\max\{\alpha,1\}}} 
	\leq C.
\end{equation*}

By applying \eqref{eq:indu1}--\eqref{eq:indu3} to \eqref{eq:fundthm}, we gather
\begin{equation*}
	\begin{split}
		&\frac{1}{2}\|\Theta_{i,j,\ell}^{q+1}(t',\tau',\omega')\|_{H^\lambda}^2 \\
		&\qquad \leq \frac{1}{2}\|\Theta_{i,j,\ell}^{q+1}(t',\tau',\frac{t_n}{2})\|_{H^\lambda}^2 
		+ C\int_{\frac{t_n}{2}}^{\omega'}
		\|\Theta_{i,j,\ell}^{q+1}(t',\tau',\tilde\omega)\|_{H^\lambda}~d\tilde\omega  \\
		&\qquad \quad+ C\int_{\frac{t_n}{2}}^{\omega'}
		\|\Theta_{i,j,\ell}^{q+1}(t',\tau',\tilde\omega)\|_{H^\lambda}^2~d\tilde\omega \\
		&\qquad \leq \frac{1}{2}\|\Theta_{i,j,\ell}^{q+1}(t',\tau',\frac{t_n}{2})\|_{H^\lambda}^2 
		+ C\Delta t + C\int_{\frac{t_n}{2}}^{\omega'}
		\|\Theta_{i,j,\ell}^{q+1}(t',\tau',\tilde\omega)\|_{H^\lambda}^2~d\tilde\omega,
	\end{split}
\end{equation*}
where the last inequality as an application of the H\"older inequality to the second term.
Now, by applying the Gronwall inequality to the previous inequality, we conclude
\begin{equation}\label{eq:firststep}
	\begin{split}
		\|\Theta_{i,j,\ell}^{q+1}(t',\tau',\omega')\|_{H^\lambda}^2
		\leq \left(\|\Theta_{i,j,\ell}^{q+1}(t',\tau',\frac{t_n}{2})\|_{H^\lambda}^2 
		+ C\Delta t\right) e^{C\Delta t}.
	\end{split}
\end{equation}

We  have now derived a bound on $\Theta_{i,j,\ell}^{q+1}(t',\tau',\omega')$ in terms 
of  $\Theta_{i,j,\ell}^{q+1}(t',\tau',\frac{t_n}{2})$. Next, we derive a bound 
on  $\Theta_{i,j,\ell}^{q+1}(t',\tau',\frac{t_n}{2})$ 
in terms of  $\Theta_{i,j,\ell}^{q+1}(t',t_n,\frac{t_n}{2})$. 
For this purpose, we once more apply the 
fundamental theorem of calculus to obtain 
\begin{equation*}
	\begin{split}
		&\frac{1}{2}\|\Theta_{i,j,\ell}^{q+1}(t',\tau',\frac{t_n}{2})\|_{H^\lambda}^2 \\
		&\quad =\frac{1}{2}\|\Theta_{i,j,\ell}^{q+1}(t',t_n,\frac{t_n}{2})\|_{H^\lambda}^2\\
		&\qquad+\sum_{s=0}^\lambda
		\int_{t_n}^{\tau'}\int_{\R^N}\Grad^s\partial_\tau
		\Theta_{i,j,\ell}^{q+1}(t',\tilde s,\frac{t_n}{2}):\Grad^s
		\Theta_{i,j,\ell}^{q+1}(t',\tilde s,\frac{t_n}{2})~dxd\tilde s \\
		&\quad = \frac{1}{2}\|\Theta_{i,j,\ell}^{q+1}(t',t_n,\frac{t_n}{2})\|_{H^\lambda}^2\\
		&\qquad+\sum_{s=0}^\lambda
		\int_{t_n}^{\tau'}\int_{\R^N}
		\Grad^s A\left(\Theta_{i,j,\ell}^{q+1}(t',\tilde s,\frac{t_n}{2})\right):
		\Grad^s\Theta_{i,j,\ell}^{q+1}(t',\tilde s,\frac{t_n}{2})~dxd\tilde s \\ 
		&\quad \leq \frac{1}{2}
		\|\Theta_{i,j,\ell}^{q+1}(t',t_n,\frac{t_n}{2})\|_{H^\lambda}^2,
	\end{split}
\end{equation*}
where we have used that $\vvt_\tau = A(\vvt)$ for $\omega= \frac{t_n}{2}$ and (4) of Definition \ref{def:operators}.
It follows that
\begin{equation}\label{eq:secondstep}
	\|\Theta_{i,j,\ell}^{q+1}(t',\tau',\frac{t_n}{2})\|_{H^k}^2 
	\leq \|\Theta_{i,j,\ell}^{q+1}(t',t_n,\frac{t_n}{2})\|_{H^k}^2.
\end{equation}

Finally, we perform our last application of the fundamental theorem
to obtain
\begin{equation*}\label{eq:fundthm2}
	\begin{split}
		&\frac{1}{2}\|\Theta_{i,j,\ell}^{q+1}(t',t_n,\frac{t_n}{2})\|_{H^\lambda}^2 
		-\frac{1}{2}\|\Theta_{i,j,\ell}^{q+1}(\frac{t_n}{2},t_n,\frac{t_n}{2})\|_{H^\lambda}^2\\
		&\qquad \qquad=\sum_{s=0}^\lambda
		\int_{\frac{t_n}{2}}^{\omega'}\int_{\R^N}
		\Grad^s\partial_t\Theta_{i,j,\ell}^{q+1}(\tilde s,t_n,\frac{t_n}{2}):
		\Grad^s\Theta_{i,j,\ell}^{q+1}(\tilde s,t_n,\frac{t_n}{2})~dxd\tilde s.
	\end{split}
\end{equation*}
By applying the same calculations to the previous 
equations as those leading to \eqref{eq:firststep}, we find that
\begin{equation}\label{eq:laststep}
	\begin{split}
		\|\Theta_{i,j,\ell}^{q+1}(t',t_n,\frac{t_n}{2})\|_{H^\lambda}^2 
		\leq \left(\|\Theta_{i,j,\ell}^{q+1}(\frac{t_n}{2},t_n,\frac{t_n}{2})\|_{H^\lambda}^2 + C \Delta t\right)e^{C\Delta t}.
	\end{split}
\end{equation}

Combining \eqref{eq:firststep}, \eqref{eq:secondstep}, and \eqref{eq:laststep} gives
\begin{equation*}
	\|\Theta_{i,j,\ell}^{q+1}(t',\tau',\omega')\|_{H^\lambda} 
	\leq
        \|\Theta_{i,j,\ell}^{q+1}(\frac{t_n}{2},t_n,\frac{t_n}{2})\|_{H^\lambda}e^{2C\Delta t} 
        + C\Delta t\left(e^{2C\Delta t} + e^{C\Delta t} \right),
\end{equation*}
which immediately leads to the conclusion
\begin{equation}\label{eq:almostdone}
	\begin{split}
		\|\Theta_{i,j,\ell}^{q+1}(t',\tau',\omega')\|_{H^\lambda} 
		&\leq \|\Theta_{i,j,\ell}^{q+1}(0,0,0)\|_{H^\lambda}e^{nC\Delta t}
		+ C\Delta t\sum_{\vr= 1}^Ne^{\vr C\Delta t}		\\
		&\leq \|\Theta_{i,j,\ell}^{q+1}(0,0,0)\|_{H^\lambda}e^{Ct'} + Ct'e^{Ct'}.
	\end{split}
\end{equation}

Finally, let us estimate $\|\Theta_{i,j,\ell}^{q+1}(0,0,0)\|_{H^\lambda}$. By definition, we have that
\begin{equation*}
	\Theta_{i,j,\ell}^{q+1}(0,0,0) =
        \frac{\partial^{q+1}}{\partial t^{i}\partial\tau^{j}\partial\omega^{\ell}}\vvt(0,0,0).
\end{equation*}
At this point, we can apply each of the time-derivatives to $\vvt(0,0,0)$ and use 
the definition \eqref{eq:strang-2} to translate time-derivatives into spatial derivatives. 
The maximal number of spatial derivatives for each time derivative is $\max\{\alpha,1\}$
and consequently $q+1$ time derivatives translates into at most $(q+1)\max\{\alpha,1\}$
spatial derivatives. As a consequence, we can conclude that
\begin{equation*}
	\|\Theta_{i,j,\ell}^{q+1}(0,0,0)\|_{H^\lambda} 
	= \|\Theta_{i,j,\ell}^{q+1}(0,0,0)\|_{H^{k-(q+1)\max\{\alpha,1\}}}\leq C\|u_0\|_{H^{k}}. 
\end{equation*}
From this, and  \eqref{eq:almostdone}, we conclude that the lemma 
holds also for $\abs{l}=q+1$. 
\end{proof}

Using the previous lemma, we now deduce time regularity of 
the source term $F$.

\begin{lemma}\label{lem:doubleF}
	Let $\vvt$ be the Strang splitting solution in 
	the sense of Definition \ref{def:strang} and \eqref{eq:strang-2}.
	Then, for any $3\max\{\alpha,1\} \leq k\in \N$  such that $\|u_0\|_{H^{k}} \leq
        C$, we have 
\begin{equation*}
	\|\nabla_{ t}^{2} F(t,\tau,\omega)\|_{H^{k-3\max\{\alpha,1\}}} \leq C, \quad (t,\tau, \omega) \in \Om_{\Delta t}.
\end{equation*}
\end{lemma}
\begin{proof}
Let $\ell = k - 3\max\{\alpha,1\}$.
Let $i$ and $j$ denote any one of $t$, $\tau$, or $\omega$. An arbitrary component 
of $\nabla^2_t F$ can then be written $F_{ij}:=\partial_i \partial_j F$. By definition, 
we have that 
\begin{equation*}
	\begin{split}
	F_{ij} &= \partial_i \partial_j \left[	\frac{1}{2}\left(\vvt_t + \Div (\vvt \vc{v}(\vvt)\right)
			 + \vvt_\tau - A(\vvt)\right]\\
	&= \frac{1}{2}\partial_t \vvt_{ij} + \partial_\tau \vvt_{ij}  - A(\vvt_{ij})
	+ \frac{1}{2}\Div \left(\vvt_{ij} \vc{v}(\vvt) + \vvt_i\vc{v}(\vvt_j)+\vvt_j\vc{v}(\vvt_i)+\vvt \vc{v}(\vvt_{ij})\right).
	\end{split}
\end{equation*}	
By applying the H\" older inequality and Sobolev embedding, we make 
the gross overestimation (note that we can have $\ell \leq 2$):
\begin{equation*}
	\begin{split}
		\|F_{ij}\|_{H^\ell} 
		& \leq \frac{3}{2}\|\nabla_{ t}^{3} \vvt\|_{H^\ell}
		+ \|\nabla_{t}^2 \vvt\|_{H^{\ell+\alpha}} 
		+ \|\Grad \vvt_{ij} \cdot \vc{v}(\vvt)\|_{H^{\ell}}
		+ \|\vvt_{ij} \Div \vc{v}(\vvt)\|_{H^{\ell}} \\
		&\quad 
		+\|\Grad \vvt_{i} \cdot \vc{v}(\vvt_j)\|_{H^{\ell}} 
		+ \|\vvt_{i} \Div \vc{v}(\vvt_j)\|_{H^{\ell}} 
		+\|\Grad \vvt_{j} \cdot \vc{v}(\vvt_i)\|_{H^{\ell}}\\
		&\quad 
		+ \|\vvt_{j} \Div \vc{v}(\vvt_i)\|_{H^{\ell}} 
		+\|\Grad \vvt \cdot \vc{v}(\vvt_{ij})\|_{H^{\ell}}
		+ \|\vvt \Div \vc{v}(\vvt_{ij})\|_{H^{\ell}} \\
		&\leq \frac{3}{2}\|\nabla_{ t}^3 \vvt\|_{H^\ell}
		+ \|\nabla_{t}^2 \vvt\|_{H^{\ell+\alpha}}
		 + C\|\nabla_t^2 \vvt\|_{H^{\ell+1}}
		 \left(\|\vc{v}(\vvt)\|_{H^{\ell+2}} + \|\Div \vc{v}(\vvt)\|_{H^{\ell+2}}\right) \\
		&\quad + 2\|\nabla_t\vvt\|_{H^{\ell+1}}\|\vc{v}(\nabla_t\vvt)\|_{H^{\ell+2}}
		+2\|\nabla_t \vvt\|_{H^{l+2}}\|\vc{v}(\nabla_t \vvt)\|_{H^{l+1}} \\
		&\quad+ C\|\vvt\|_{H^{\ell+3}}\left(\|\vc{v}(\nabla_t^2\vvt )\|_{H^\ell} 
		+\|\Div \vc{v}(\nabla_t^2\vvt )\|_{H^\ell}\right)\\
		&\leq C\left( \|\nabla_{ t}^3 \vvt\|_{H^\ell}+\|\nabla_{t}^2 \vvt\|_{H^{\ell+\max\{\alpha,1\}}} 
		+ \|\nabla_t \vvt\|_{H^{\ell+2}} \right).
	\end{split}
\end{equation*}
Since $\vvt(0,0) := u_0 \in H^k$, Lemma \ref{lem:tempv} tell 
us that the right-hand side is bounded and hence our proof is complete.
\end{proof}

The next lemma is the main property
giving second order convergence. Observe that the result does 
not depend on our specific choice of operators $A$ and $B$.
\begin{lemma}\label{lem:magic}
	There holds
	\begin{equation*}
		\nabla_{ t}F(\frac{t_n}{2},t_n,\frac{t_n}{2})\cdot
		\begin{pmatrix}
			1 \\ 2 \\ 1
		\end{pmatrix}
		= 0.
	\end{equation*}	
\end{lemma}

\begin{proof}
We will prove Lemma \ref{lem:magic} by direct calculation. Let us 
begin by estimating $F_\omega(\frac{t_n}{2},t_n,\frac{t_n}{2})$. 
Since $F(t,t_n,\frac{t_n}{2})= 0$,  we have that
\begin{equation}\label{eq:easy1}
	\begin{split}
		F_\omega(\frac{t_n}{2},t_n,\frac{t_n}{2}) = - [A,B](\vvt).
	\end{split}
\end{equation}
Similarly, we see that \eqref{eq:feqS} yields
\begin{equation}\label{eq:easy2}
	\begin{split}
		F_\tau = \frac{1}{2}[A,B](\vvt).
	\end{split}
\end{equation}
It only remains to estimate $F_t(\frac{t_n}{2},t_n,\frac{t_n}{2})$. 
However, as $F(t,t_n,\frac{t_n}{2}) = 0$,  we must 
have that $F_t(\frac{t_n}{2},t_n,\frac{t_n}{2})= 0$. 
This, together with \eqref{eq:easy1} 
and \eqref{eq:easy2}, concludes the proof.
\end{proof}

Using the previous lemma, we can now prove that 
the error produced along the diagonal $(t/2,t,t/2)$
is second order in $\Delta t$. 
\begin{lemma}\label{lem:spoton}
	Let $\vvt$ be the Strang splitting solution in the sense 
	of Definition \ref{def:strang} and \eqref{eq:strang-2}. 
	Then, if $u_0 \in H^k$,
	\begin{equation*}
		\left\|F\left(\frac{t}{2},t,\frac{t}{2}\right)\right\|_{H^{k-3\max\{\alpha,1\}}} 
		\leq C(\Delta t)^2.
	\end{equation*}
\end{lemma}
\begin{proof}
Since $F(\frac{t_n}{2},t_n,\frac{t_n}{2}) = 0$, a Taylor expansion provides the identity
\begin{equation}\label{eq:jeezus}
	\begin{split}
		F\left(\frac{t}{2},t,\frac{t}{2}\right)  
		&= \nabla_{t}F(\frac{t_n}{2},t_n,\frac{t_n}{2})\cdot
		\begin{pmatrix}
			1 \\ 2 \\ 1
		\end{pmatrix}(\frac{t}{2}-\frac{t_n}{2}) \\
		&\qquad\qquad 
		+\frac{1}{2}\int_{t_n/2}^{t/2}
		\begin{pmatrix}
			1 \\ 2 \\ 1
		\end{pmatrix}^T
		\nabla_{t}^2 F \left(\frac{s}{2},s,\frac{s}{2}\right)
		\begin{pmatrix}
			1 \\ 2 \\ 1
		\end{pmatrix}(\frac{s}{2}-\frac{t_n}{2})~ds\\
		&= \frac{1}{2}\int_{t_n/2}^{t/2}
		\begin{pmatrix}
			1 \\ 2 \\ 1
		\end{pmatrix}^T
		\nabla_{t}^2 F \left(\frac{s}{2},s,\frac{s}{2}\right)
		\begin{pmatrix}
			1 \\ 2 \\ 1
		\end{pmatrix}(\frac{s}{2}-\frac{t_n}{2})~ds,
		\end{split}
	\end{equation}
	where the last equality is an application of Lemma \ref{lem:magic}.
	By taking the $H^{k-3\max\{\alpha,1\}}$ norm on both 
	sides of \eqref{eq:jeezus} and applying the 
	previous lemma, we gather
	\begin{equation*}
		\left\|F(\frac{t}{2},t,\frac{t}{2})\right\|_{H^k} \leq C(\Delta t)^2,
	\end{equation*}
	which concludes the proof.
\end{proof}

\begin{proof}[Proof of Lemma \ref{lem:rate}]
By performing the same calculations as those in the proof of Lemma \ref{lem:godunoverror}
with the new estimate on $F$ given by Lemma \ref{lem:spoton}, we obtain the estimate
\begin{equation*}
	\frac{1}{2}\partial_t\|e(t)\|_{H^{k-3\max\{\alpha,1\}}} \leq C(T)\left((\Delta t)^2 
	+ \|e(t)\|_{H^{k-3\max\{\alpha,1\}}}\right), \qquad t \in (0,T).
\end{equation*}
Since $e(0)=0$, an application of the Gronwall inequality to the previous inequality gives
\begin{equation*}
	\|e(t)\|_{H^{k-3\max\{\alpha,1\}}} \leq t\, (\Delta t)^2C, \quad t \in [0,T],
\end{equation*}
which concludes the proof of Lemma \ref{lem:rate} and 
consequently also Theorem \ref{thm:strang}.
\end{proof}

\section{Applications}\label{sec:applications}
In the previous sections we have established well-posedness and convergence 
rates for both Godunov and Strang splitting applied to \eqref{eq:eq}.
In this section we will examine a range of different equations that are 
all of the form \eqref{eq:eq} with $A$ and $\vc{v}$ being 
admissible (in the sense of Definition \ref{def:operators}). Our primary 
goal is to equip the reader with some relevant applications of our framework.
Let us for the convenience of the reader repeat our  
equation and assumptions here.
We are working with the equation
\begin{equation*}
	u_t + \Div \left(u\,\vc{v}(u)\right) 
	= A(u), \qquad u|_{t=0} = u_0,
\end{equation*}
where the operators are required to be admissible in 
the  sense of Definition \ref{def:operators}.

Given a specific equation, the only non-trivial conditions in Definition \ref{def:operators}  are
the well-posedness (6) and the commutator estimate (5).
The following lemma is of help to determine the latter.
\begin{lemma}\label{lem:examples}
Let $\alpha \in (0,2)$ and $l$ be any multi-index. The standard differential operator
$A(u)= D^lu$
and the fractional Laplacian
$A(u)= (-\Delta )^{\alpha/2} u$
both satisfy  
(3)--(5) in Definition \ref{def:operators}.
\end{lemma}
\begin{proof}
Conditions (3) and (4) are trivially satisfied. It remains to prove
that $A$ satisfies the commutator estimate (5). In the case of the fractional 
Laplacian ($A(u) = (-\Delta)^{\alpha/2}u$), the commutator 
estimate was proved in Corollary 2.5 in \cite{Holden:karper}. See also 
the lecture notes by Constantin \cite{Constantin1}.
For the standard differential operator, the result is immediate 
from the standard Leibniz rule.
\end{proof}

\begin{remark}
By combining the usual Leibniz rule with the corresponding rule for 
the fractional Laplacian, the previous lemma can be extended to also include operators 
on the form $A(u) = (-\Delta)^{\alpha/2}D^l u$, that is, any 
mix of fractional and standard derivatives. 
\end{remark}

\subsection{Burgers type equations}
If we restrict to one spatial dimension $(N=1)$, the only valid 
velocity operator $\vc{v}$ satisfying requirement (1) of Definition \ref{def:operators}
is
\begin{equation*}
	\vc{v}(u)= au, \qquad a \in \R.
\end{equation*}
Thus, the type of equations we can consider consists of a Burgers 
term and a linear differential term. In the literature one 
can find several equations of this type that are  well-posed
in the sense of (6) in Definition \ref{def:operators}. 
Some examples are:
\begin{align*}
  u_t + (u^2)_x &= u_{xxx}, & (\text{{\sc{KdV}}}), \\
	u_t + (u^2)_x &= -(-\partial_x^2)^{{\alpha/2}}u,& (\text{viscous Burgers}), \\
	u_t + (u^2)_x &= -u_{xxx} + u_{xxxxx}, & (\text{Kawahara}).
\end{align*}
For the viscous Burgers equation, $\alpha \geq 1$ is required to ensure 
well-posedness \cite{Kiselev3}.

\subsection{Quasi-geostrophic flow}
The following 
equation has been proposed  as a toy model for strongly
rotating atmospheric flow
\begin{equation}\label{eqL:qg}
	u_t +  \Div (u \,\vc{v}(u)) = 0, \quad \vc{v}(u) = \Curl (-\Delta)^{-1/2}u,
\end{equation}
where $u$ is the potential temperature, $\vc{v}(u)$ is the fluid velocity, and 
the equation is valid in two dimensions $(N=2)$.
The reader can consult \cite{Constantin1,Constantin4} and the references therein for more 
on the physical aspects of the model. The equation \eqref{eqL:qg}
is in the literature referred to as the \emph{quasi-geostrophic equation} 
and has in the recent years been the subject of numerous analytical studies. 
This recent interest was probably sparked by the  Constantin, Majda, and Tabak, paper 
\cite{Constantin4} in which they give numerical evidence for a 
 connection between the blow-up of solutions to \eqref{eqL:qg} and the 
three dimensional Euler equations.  Though the precise type of 
blow-up has been later dismissed by D. Cordoba \cite{Cordoba2},
there remains hope that understanding the behavior of solutions to \eqref{eqL:qg} can 
aid in characterizing blow-up of solutions to the incompressible Euler equations, a
long standing open problem. As a consequence, most of the recent studies 
concerns regularity of solutions to equations of 
the form \eqref{eqL:qg}. A particularly well-studied 
case is the dissipative quasi-geostrophic equation
\begin{equation}\label{eqL:dqg}
	u_t +   \Div (u \,\vc{v}(u)) = A(u),\quad \vc{v}(u) = \Curl (-\Delta)^{-\beta/2}u, 
\end{equation}
where $A(u) = -(-\Delta)^{\alpha/2}u$. Through the 
papers \cite{Caffarelli1, Constantin2, Cordoba, Kiselev2}, it is has been 
shown that \eqref{eqL:dqg} is well-posed in 
the sense of (6) in Definition \ref{def:operators} for 
$\alpha$, $\beta$ $\in [1,2]$. Since in addition 
$\Div \vc{v}(u)=0$,  (2) in Definition \ref{def:operators}
is also satisfied. Hence, \eqref{eqL:dqg} 
with $\alpha$, $\beta \in [1,2]$ is admissible in our framework. 

\subsection{Aggregation equations}
The following active scalar equation
has been proposed \cite{Mogilner1, Bertozzi1, Bertozzi2} as a model for the
long-range attraction between individuals in flocks, schools, or swarms:
\begin{equation}\label{eq:ag2}
	u_t + \Div \left(u\, \vc{v}(u)\right)  = 0.
\end{equation}
The velocity is determined as the convolution 
of $u$ with an interaction potential:
$$
\vc{v}(u) =  \nabla \Phi \star u.
$$
In these models, $u$ is the density of individuals and 
$\vc{v}(u)$ incorporates the pairwise attractive 
forces between individuals in the flock. Two relevant 
examples of potentials $\Phi$ in applications are 
the radially symmetric 
$\Phi = 1-e^{-|x|}$ and $\Phi = 1-e^{-|x|^2}$.

In the papers \cite{Bertozzi-Carrillo, Bertozzi-Laurent1}, sharp  
conditions are derived on the potential $\Phi$ under which 
solutions of \eqref{eq:ag2} blows up in finite time.
In particular, whenever
\begin{equation*}\label{eq:cond}
	\int_0^1 \frac{1}{\Phi'(r)}~dr < \infty,
\end{equation*}
solutions of \eqref{eq:ag2} collapse to a point  
at the center of mass in finite time. 
Thus, the relevant case $\Phi = 1-e^{-|x|}$ leads to blow up 
in finite time. This clearly non-realistic behavior 
can be attributed to the lack of any effect incorporating 
collision avoidance in the model. Since collision avoidance 
is a short-range phenomena, a simple 
 way to incorporate it is to add 
diffusion to the model
\begin{equation}\label{eq:dag}
	u_t + \Div \left(u\, \vc{v}(u)\right)  = A(u), \qquad A(u) 
	= -(-\Delta)^{\alpha/2}u.
\end{equation}
The equation \eqref{eq:dag} has been the subject of 
the recent studies \cite{DongAG1, DongAG2, DongAG3} 
and well-posedness in the sense of (6) in Definition \ref{def:operators} has been 
established for $\alpha \in (1,2]$.
In view of this well-posedness result and Lemma \ref{lem:examples},
\eqref{eq:dag} is included in our framework provided 
we can verify condition (2) in Definition \ref{def:operators}.

By applying the H\"older inequality, we see that
\begin{equation*}
	\begin{split}
		\int_{\R^N}\Div \vc{v}(f)\phi~dx
		= \int_{\R^N}\int_{\R^N}\Delta \Phi(y)f(x-y)\phi(x)~dxdy 
		\leq \|f\|_{L^p}\|\phi\|_{L^{p'}}\|\Delta \Phi\|_{L^1}.
	\end{split}
\end{equation*}
Hence, $\Div \vc{v}(f)$ satisfies (2) in Definition \ref{def:operators}
for any potential $\Phi$ satisfying $\Delta \Phi \in L^1$.
Comparing this condition with 
the results in the paper \cite{Laurent}, we see that 
we can include all the potentials for
which well-posedness is known.
For instance, the case $\Phi = 1-e^{-|x|}$, since
\begin{equation*}
	\begin{split}
		\Delta \Phi 
		= \Div \left(\frac{x}{|x|}e^{-|x|}\right) 
		= -e^{-|x|} + \frac{N-1}{|x|}e^{-|x|}, \quad x \neq 0,
	\end{split}
\end{equation*}
and
\begin{equation*}
	\int_{\{|x| \leq 1\}} \frac{N-1}{|x|}~dx \leq C, \qquad N \geq 2.
\end{equation*}

\subsection{Magneto geostrophic dynamics}
In the recent paper \cite{Friedlander}, well-posedness was established for 
the following class of active scalar equations 
\begin{equation}\label{eq:mag}
	u_t + \Div(u\,\vc{v}(u))  = A(u), \quad \vc{v}(u) = \Div \mathbb{T} (u),
\end{equation}
where $\mathbb{T}$ is a matrix of Calderon--Zygmund operators 
satisfying 
$$
\Div \vc{v}(u) = \Div \Div \mathbb{T}(u) = 0,
$$and $A$ is the Laplace operator $A(u) = \Delta u$.
The equation \eqref{eq:mag} can be seen as a generalization of 
the quasi-geostrophic equation \eqref{eqL:dqg} (with $\alpha = 2$) since the latter can be 
obtained from \eqref{eq:mag} with a particular choice 
of $\mathbb{T}$ in 2D. The equation \eqref{eq:mag} is also included 
in the framework considered by Constantin in \cite{Constantin3}.

The physical motivation for the study \cite{Friedlander} was that 
\eqref{eq:mag} with a particular choice of $\mathbb{T}$ has 
been proposed as a model for magnetostrophic 
turbulence in the Earth's fluid core. See the 
paper \cite{Friedlander} and the references therein 
for more on this application. Since  \eqref{eq:mag} is well-posed 
in the sense of (6) in Definition \ref{def:operators},  our framework 
does indeed include this class of active scalar equations.

\appendix

\section{Proof of Lemmas \ref{lem:nonlineardiv} 
and \ref{lem:lineardiv}}\label{sec:appendix}
In this appendix we have gathered the proofs 
of Lemmas \ref{lem:nonlineardiv} and \ref{lem:lineardiv}. 
Both  lemmas have been used in an essential 
fashion throughout the convergence analysis. 

\renewcommand{\thesection}{2}

\begin{lemma}\label{lem:nonlineardivA}
Let $k\geq 6$ and  $\vc{v}$ be an operator that 
satisfies Definition \ref{def:operators}. Then,
\begin{equation*}
	\sum_{s=0}^k\left|\int_{\R^N} \Grad^s(\Div (f \vc{v}(f))):
	\Grad^s f~  dx\right| \leq C\|f\|_{H^{k-2}}\|f\|_{H^k}^2,
	\quad  f \in H^k.
\end{equation*}	
\end{lemma}
\begin{proof}
Let us confine to the three dimensional case ($N=3$) 
as the other cases are almost identical. 
The product rule provides  us with the identity
\begin{equation*}
	\begin{split}
	&\sum_{s= 0}^k\left|\int_{\R^N}\Grad^s 
	\Div \left(f \vc{v}(f)\right):\Grad^s f~  dx\right| \\
	&\qquad \leq \sum_{s=0}^k\Big(\left|\int_{\R^N}
	\Grad^s\left(\Grad f \cdot \vc{v}(f) \right): \Grad^sf~dx\right| 
	+ \left|\int_{\R^N}\Grad^s (f \Div \vc{v}(f)): \Grad^sf~dx \right| \Big)\\
	&\qquad := \sum_{s=0}^k\left|I_1^s\right| + \left|I_2^s\right|.		
	\end{split}
\end{equation*}
In the remaining parts of the proof, our strategy is to bound 
each of $I_1^s$ and $I_2^s$, $s=0, \ldots, k$, separately.
We begin with $I_1^s$.

1. By applying the  Leibniz 
rule to $I_1^s$ (with multi-index notation $\alpha = (\alpha_1, \alpha_2, \alpha_3)$), we obtain the following expression 
\begin{align}\label{eq:I1start}
	I_1^s &= \sum_{|\alpha| = s}\int_{\R^N} 
	\Grad^\alpha(\Grad f \cdot \vc{v}(f))\Grad^\alpha f~dx \nonumber\\
	&=  \sum_{|\alpha| = s}\sum_{i_1=0}^{\alpha_1}\sum_{i_2=0}^{\alpha_2}
	\sum_{i_3=0}^{\alpha_3}
	\left(\alpha_1 \atop i_1\right)\left(\alpha_2 \atop i_2\right)\left(\alpha_3 \atop i_3\right) \\
	&\quad \times \int_{\R^N} \left(
	\Grad \frac{\partial^{i_1 + i_2 + i_3}f}{\partial x^{i_1}\partial y^{i_2}\partial z^{i_3}} 
	\cdot \vc{v}\left(\frac{\partial^{s-i_1 - i_2 - i_3}f}{\partial x^{\alpha_1 - i_1}\partial y^{\alpha_2 - i_2}
	\partial z^{\alpha_3 - i_3}}\right) \right)
	\frac{\partial^s f}{\partial x^{\alpha_1}\partial y^{\alpha_2}
	\partial z^{\alpha_3}}~dx. \nonumber
\end{align}
Let us now consider four separate cases of $i_1+i_2+i_3$ in the above quadruple sum.

(i) If $i_1+i_2+ i_3= k=s$, i.e.,  $(\alpha_1, \alpha_2, \alpha_3)=(i_1,i_2,i_3)$,  the 
above term can be rewritten as follows
\begin{equation*}
	\begin{split}
		&\int_{\R^N} \left(\Grad \frac{\partial^k f}{\partial x^{\alpha_1}
		\partial y^{\alpha_2}
		\partial z^{\alpha_3}} \cdot \vc{v}(f)\right)
		\frac{\partial^k f}{\partial x^{\alpha_1}
		\partial y^{\alpha_2}\partial z^{\alpha_3}}~dx \\
		&=
		\int_{\R^N} \frac{1}{2}\Grad \left|\frac{\partial^k f}{\partial x^{\alpha_1}
		\partial y^{\alpha_2}\partial z^{\alpha_3}}\right|^2
		\cdot \vc{v}(f)~  dx 
		  = -\int_{\R^N} \frac{1}{2} \left|\frac{\partial^k f}{\partial x^{\alpha_1}
		  \partial y^{\alpha_2}\partial z^{\alpha_3}}\right|^2
		\Div \vc{v}(f)~  dx \\
		& \leq \frac{1}{2}\left\|\frac{\partial^k f}{\partial x^{\alpha_1}
		\partial y^{\alpha_2}\partial z^{\alpha_3}}\right\|^2_{L^2}
		\|\Div \vc{v}(f)\|_{L^\infty}
		\leq C\|f\|_{H^k}^2\|\vc{v}(f)\|_{H^3} 
		\leq  C\|f\|^2_{H^{k}}\|f\|_{H^{k-2}},
	\end{split}
\end{equation*}
where we have  used the Sobolev embedding $H^2 \subset L^\infty$ and $k \geq 6$.

(ii) If $2 \leq i_1 + i_2 + i_3  \leq k-3$, we find
\begin{equation*}
	\begin{split}
		&\int_{\R^N} \left(\Grad \frac{\partial^{i_1 + i_2 + i_3}f}{\partial x^{i_1}
		\partial y^{i_2}\partial z^{i_3}} 
		\cdot \vc{v}\left(\frac{\partial^{s-i_1 - i_2 - i_3}f}{\partial x^{\alpha_1 - i_1}
		\partial y^{\alpha_2 - i_2}\partial z^{\alpha_3 - i_3}}\right) \right)
		\frac{\partial^s f}{\partial x^{\alpha_1}\partial y^{\alpha_2}\partial z^{\alpha_3}}~dx \\
		&\leq \left\|\Grad \frac{\partial^{i_1 + i_2 + i_3}f}{\partial x^{i_1}
		\partial y^{i_2}\partial z^{i_3}} \right\|_{L^\infty}
		\left\|\vc{v}\left(\frac{\partial^{s-i_1 - i_2 - i_3}f}{\partial x^{\alpha_1 - i_1}
		\partial y^{\alpha_2 - i_2}\partial z^{\alpha_3 - i_3}}\right)\right\|_{L^2}
		\left\|\frac{\partial^s f}{\partial x^{\alpha_1}
		\partial y^{\alpha_2}\partial z^{\alpha_3}}\right\|_{L^2}
		\\ & \leq C\|f\|_{H^{k}}\|f\|_{H^{k-2}}\|f\|_{H^{k}}.
	\end{split}
\end{equation*}

(iii) If $k-2 \leq i_1 + i_2 + i_3 \leq k-1$, we find
\begin{equation*}
	\begin{split}
		&\int_{\R^N} \left(\Grad \frac{\partial^{i_1 
		+ i_2 + i_3}f}{\partial x^{i_1}\partial y^{i_2}\partial z^{i_3}} 
		\cdot \vc{v}\left(\frac{\partial^{s-i_1 - i_2 - i_3}f}{\partial x^{\alpha_1 - i_1}
		\partial y^{\alpha_2 - i_2}\partial z^{\alpha_3 - i_3}}\right) \right)
		\frac{\partial^s f}{\partial x^{\alpha_1}\partial y^{\alpha_2}
		\partial z^{\alpha_3}}~dx \\
		&\leq \left\|\Grad \frac{\partial^{i_1 + i_2 + i_3}f}{\partial x^{i_1}
		\partial y^{i_2}\partial z^{i_3}}\right\|_{L^2}
		\left\|\vc{v}\left(\frac{\partial^{s-i_1 - i_2 - i_3}f}{\partial x^{\alpha_1 - i_1}
		\partial y^{\alpha_2 - i_2}\partial z^{\alpha_3 - i_3}}\right)\right\|_{L^\infty}
		\left\|\frac{\partial^s f}{\partial x^{\alpha_1}\partial y^{\alpha_2}
		\partial z^{\alpha_3}}\right\|_{L^2} \\
		&\leq C\|f\|_{H^{k}}\|f\|_{H^4}\|f\|_{H^{k}} 
		\leq C\|f\|_{H^{k-2}}\|f\|_{H^{k}}^2.
	\end{split}
\end{equation*}
To conclude the second last inequality, we have used that $k-2 \leq s \leq k$ and hence that 
$s-i_1-i_2-i_3 \leq 2$. The inequality then follows from the embedding $H^2 \subset L^\infty$.

(iv) If $0 \leq i_1 + i_2 + i_3 \leq 1$, we find
\begin{equation*}
	\begin{split}
		&\int_{\R^N} \left(\Grad \frac{\partial^{i_1 + i_2 + i_3}f}{\partial x^{i_1}
		\partial y^{i_2}\partial z^{i_3}} 
		\cdot \vc{v}\left(\frac{\partial^{s-i_1 - i_2 - i_3}f}{\partial x^{\alpha_1 - i_1}
		\partial y^{\alpha_2 - i_2}\partial z^{\alpha_3 - i_3}}\right) \right)
		\frac{\partial^s f}{\partial x^{\alpha_1}\partial y^{\alpha_2}
		\partial z^{\alpha_3}}~dx \\
		&\leq \left\|\Grad \frac{\partial^{i_1 + i_2 + i_3}f}{\partial x^{i_1}\partial y^{i_2}
		\partial z^{i_3}}\right\|_{L^\infty}
		\left\|\vc{v}\left(\frac{\partial^{s-i_1 - i_2 - i_3}f}{\partial x^{\alpha_1 - i_1}
		\partial y^{\alpha_2 - i_2}\partial z^{\alpha_3 - i_3}}\right)\right\|_{L^2}
		\left\|\frac{\partial^s f}{\partial x^{\alpha_1}\partial y^{\alpha_2}
		\partial z^{\alpha_3}}\right\|_{L^2} \\ & 
		\leq C\|f\|_{H^4}\|f\|_{H^{k}}^2 \leq C\|f\|_{H^{k-2}}\|f\|_{H^{k}}^2.
	\end{split}
\end{equation*}

Hence, by applying (i)--(iv) in \eqref{eq:I1start}, we see that
\begin{equation}\label{eq:I1bound}
	\sum_{s=0}^{k} |I_1^s|  \leq C\|f\|_{H^{k-2}}\|f\|_{H^{k}}^2.
\end{equation}

2. We now bound the $I_2^s$ terms. The standard Leibniz rule 
provides the identity
\begin{align}\label{eq:I2start}
	I_2^s &= \sum_{|\alpha| = s}\int_{\R^N} \Grad^\alpha
	\left(f \Div \vc{v}(f)\right)\Grad^\alpha f~dx \nonumber\\
	&= \sum_{|\alpha| = s}\sum_{i_1=0}^{\alpha_1}
	\sum_{i_2=0}^{\alpha_2}\sum_{i_3=0}^{\alpha_3}
	\left(\alpha_1 \atop i_1\right)\left(\alpha_2 \atop i_2\right)
	\left(\alpha_3 \atop i_3\right) \\
	&\quad \times \int_{\R^N} \left( \frac{\partial^{i_1 + i_2 + i_3}f}{\partial x^{i_1}
	\partial y^{i_2}\partial z^{i_3}} 
	\Div\vc{v}\left(\frac{\partial^{s-i_1 - i_2 - i_3}f}{\partial x^{\alpha_1 - i_1}
	\partial y^{\alpha_2 - i_2}\partial z^{\alpha_3 - i_3}}\right)\right)
	\frac{\partial^s f}{\partial x^{\alpha_1}\partial y^{\alpha_2}
	\partial z^{\alpha_3}}~dx. \nonumber
\end{align}
Let us again consider four separate cases of $i_1 + i_2 + i_3$.

(i) If $i_1+i_2+i_3 = 0$, we apply the H\"older inequality 
and (3) in Definition \ref{def:operators}, which yields to deduce
\begin{equation*}
	\begin{split}
		&\int_{\R^N} \left(f
		\Div\vc{v}\left(\frac{\partial^{s}f}{\partial x^{\alpha_1}
		\partial y^{\alpha_2}\partial z^{\alpha_3}}\right)\right)
		\frac{\partial^s f}{\partial x^{\alpha_1}\partial y^{\alpha_2}\partial z^{\alpha_3}}~dx \\
		&\leq \|f\|_{L^\infty}\left\|\Div\vc{v}\left(\frac{\partial^{s}f}{\partial x^{\alpha_1}
		\partial y^{\alpha_2}\partial z^{\alpha_3}}\right)\right\|_{L^2} \|f\|_{H^s}
		\leq C\|f\|_{H^{k-2}}\|f\|_{H^k}^2.
	\end{split}
\end{equation*}

(ii) If $3 \leq i_1 + i_2 + i_3 \leq k-2$, we find
\begin{equation*}
	\begin{split}
		&\int_{\R^N} \left(\frac{\partial^{i_1 + i_2 + i_3}f}{\partial x^{i_1}\partial y^{i_2}\partial z^{i_3}} 
		\Div\vc{v}\left(\frac{\partial^{s-i_1 - i_2 - i_3}f}{\partial x^{\alpha_1 - i_1}
		\partial y^{\alpha_2 - i_2}\partial z^{\alpha_3 - i_3}}\right)\right)
		\frac{\partial^s f}{\partial x^{\alpha_1}\partial y^{\alpha_2}\partial z^{\alpha_3}}~dx \\
		&\leq 
		\left\|\frac{\partial^{i_1 + i_2 + i_3}f}{\partial x^{i_1}\partial y^{i_2}\partial z^{i_3}}
		\right\|_{L^2} \left\|\Div\vc{v}
		\left(\frac{\partial^{s-i_1 - i_2 - i_3}f}{\partial x^{\alpha_1 - i_1}
		\partial y^{\alpha_2 - i_2}\partial z^{\alpha_3 - i_3}}\right)\right\|_{L^\infty}
		\left\|\frac{\partial^s f}{\partial x^{\alpha_1}\partial y^{\alpha_2}
		\partial z^{\alpha_3}}\right\|_{L^2} \\
		&\leq C \|f\|_{H^{k-2}}\|\vc{v}(f)\|_{H^{s}}\|f\|_{H^s} 
		\leq C\|f\|_{H^{k-2}}\|f\|_{H^{k}}^2.
	\end{split}
\end{equation*}

(iii) If $1 \leq i_1 + i_2 + i_3 \leq 2$, we find
\begin{equation*}
	\begin{split}
		&\int_{\R^N} \left( \frac{\partial^{i_1 + i_2 + i_3}f}{\partial x^{i_1}
		\partial y^{i_2}\partial z^{i_3}} 
		\Div\vc{v}\left(\frac{\partial^{s-i_1 - i_2 - i_3}f}{\partial x^{\alpha_1 - i_1}
		\partial y^{\alpha_2 - i_2}\partial z^{\alpha_3 - i_3}}\right)\right)
		\frac{\partial^s f}{\partial x^{\alpha_1}\partial y^{\alpha_2}
		\partial z^{\alpha_3}}~dx \\
		&\leq 
		\left\|\frac{\partial^{i_1 + i_2 + i_3}f}{\partial x^{i_1}\partial y^{i_2}
		\partial z^{i_3}}\right\|_{L^\infty}
		\left\|\Div\vc{v}\left(\frac{\partial^{s-i_1 - i_2 - i_3}f}{\partial x^{\alpha_1 - i_1}
		\partial y^{\alpha_2 - i_2}\partial z^{\alpha_3 - i_3}}\right)\right\|_{L^2}
		\left\|\frac{\partial^s f}{\partial x^{\alpha_1}
		\partial y^{\alpha_2}\partial z^{\alpha_3}}\right\|_{L^2} \\
		& \leq C\|f\|_{H^4}\|\vc{v}(f)\|_{H^s}\|f\|_{H^s} \leq C\|f\|_{H^{k-2}}\|f\|_{H^{k}}^2.
	\end{split}
\end{equation*}

(iv) If $k-1 \leq i_1+ i_2 + i_3 \leq k$, we find
\begin{equation*}
	\begin{split}
		&\int_{\R^N} \left( \frac{\partial^{i_1 + i_2 + i_3}f}{\partial x^{i_1}
		\partial y^{i_2}\partial z^{i_3}} 
		\Div\vc{v}\left(\frac{\partial^{s-i_1 - i_2 - i_3}f}{\partial x^{\alpha_1 - i_1}
		\partial y^{\alpha_2 - i_2}\partial z^{\alpha_3 - i_3}}\right)\right)
		\frac{\partial^s f}{\partial x^{\alpha_1}\partial y^{\alpha_2}
		\partial z^{\alpha_3}}~dx \\
		& \leq 
		\left\|\frac{\partial^{i_1 + i_2 + i_3}f}{\partial x^{i_1}
		\partial y^{i_2}\partial z^{i_3}}\right\|_{L^2}
		\left\|\Div\vc{v}\left(\frac{\partial^{s-i_1 - i_2 - i_3}f}{\partial x^{\alpha_1 - i_1}
		\partial y^{\alpha_2 - i_2}\partial z^{\alpha_3 - i_3}}\right)\right\|_{L^\infty}
		\left\|\frac{\partial^s f}{\partial x^{\alpha_1}
		\partial y^{\alpha_2}\partial z^{\alpha_3}}\right\|_{L^2} \\
		&\leq C\|f\|_{H^{k}} \|\vc{v}(f)\|_{H^4}\|f\|_{H^s} 
		\leq C\|f\|_{H^{k-2}}\|f\|_{H^{k}}^2.
	\end{split}
\end{equation*}

Applying (i)--(iv) in \eqref{eq:I2start} gives
\begin{equation*}
	\sum_{s=0}^{k}|I_2^s| \leq C\|f\|_{H^{k-2}}\|f\|_{H^{k}}^2.
\end{equation*}
Together with \eqref{eq:I1bound}, this 
brings our proof to an end.
\end{proof}

\begin{lemma}\label{lem:lineardivA}
Let $k\geq 4$. Then the following estimates hold
\begin{equation}\label{eq:ldiv12}
	\sum_{s=0}^k\left|\int_{\R^N} \Grad^s\Div \left(f \vc{v}(g)\right):
	\Grad^s f~dx\right| \leq C\|g\|_{H^k}\|f\|_{H^k}^2,
\end{equation}
\begin{equation}\label{eq:ldiv22}
	\sum_{s=0}^k\left|\int_{\R^N}\Grad^s\Div \left(g\vc{v}(f)\right):
	\Grad^s f~dx\right| \leq C\|g\|_{H^{k+1}}\|f\|_{H^k}^2.
\end{equation}
\end{lemma}

\begin{proof}
The proof of \eqref{eq:ldiv12} is easily obtained by the calculations
of the previous proof. To prove \eqref{eq:ldiv22}, it is only step (i)
in part 1 of the previous proof that does not go through. Instead, we 
now make the calculation
\begin{equation*}
	\begin{split}
		&\int_{\R^N} \left(\Grad \frac{\partial^k g}{\partial x^{\alpha_1}
		\partial y^{\alpha_2}\partial z^{\alpha_3}} \cdot \vc{v}(f)\right)
		\frac{\partial^k f}{\partial x^{\alpha_1}\partial y^{\alpha_2}
		\partial z^{\alpha_3}}~dx \\
		& \leq \left\|\Grad \frac{\partial^k g}{\partial x^{\alpha_1}
		\partial y^{\alpha_2}\partial z^{\alpha_3}}\right\|_{L^2}
		\left\|\vc{v}(f)\right\|_{L^\infty}\|f\|_{H^k}
		\leq C\|f\|_{H^{k+1}}^2\|f\|_{H^2},
	\end{split}
\end{equation*}
and the proof is complete. 
\end{proof}


\begin{thebibliography}{10}

\bibitem{Bertozzi-Carrillo}
A.~L. Bertozzi, J.~A. Carrillo, and T. Laurent.
\newblock Blow-up in multidimensional aggregation equations with mildly
              singular interaction kernels.
\newblock {\em Nonlinearity},  {\bf 22}(3):683--710, 2009.

\bibitem{Bertozzi-Laurent1}
A.~L. Bertozzi and T. Laurent.
\newblock Finite-time blow-up of solutions of an aggregation equation in {$\bold R^n$}.
\newblock {\em Comm. Math. Phys.},  {\bf 274}(3):717--735, 2007.

\bibitem{Caffarelli1}
L. Caffarelli and A. Vasseur.
\newblock Drift diffusion equations with fractional diffusion and the quasi-geostrophic equation.
\newblock {\em Ann. Math.},  {\bf 171}(3):1903--1930,  2010.


\bibitem{Constantin1}
P. Constantin.
\newblock Euler equations, {N}avier--{S}tokes equations and turbulence.
\newblock In {\em Mathematical Foundation of Turbulent Viscous Flows.}
\newblock Lecture Notes in Math., Springer, Vol. 1871, pp. 1--43, 2006.

\bibitem{Constantin3}
P. Constantin.
\newblock Scaling exponents for active scalars.
\newblock {\em J. Statist. Phys.}, {\bf 90}(3-4):571--595, 1998.

\bibitem{Constantin2}
P. Constantin and J. Wu.
\newblock Behavior of solutions of 2D quasi-geostrophic equations.
\newblock {\em SIAM J. Math. Anal.}, {\bf 30}(5):937--948, 1999.

\bibitem{Constantin4}
P. Constantin, A. Majda, and E. Tabak.
\newblock Formation of strong fronts in the {$2$}-{D} quasigeostrophic thermal active scalar.
\newblock {\em Nonlinearity}, {\bf 7}(6):1495--1533, 1994.

\bibitem{Cordoba}
A. C{\'o}rdoba and D. C{\'o}rdoba.
\newblock A maximum principle applied to quasi-geostrophic equations.
\newblock {\em Comm. Math. Phys.}, {\bf 249}(3):511--528, 2004.

\bibitem{Cordoba2}
D. C{\'o}rdoba.
\newblock Nonexistence of simple hyperbolic blow-up for the quasi-geostrophic equation.
\newblock {\em Ann. of Math.}, {\bf 148}(2):1135--1152, 1998.


\bibitem{Dong}
H. Dong and D. Du.
\newblock Global well-posedness and a decay estimate 
for the critical dissipative quasi-geostrophic equation in the whole space.
\newblock {\em Discrete Contin. Dyn. Syst.}, {\bf 21}(4):1095--1101, 2008.

\bibitem{DongAG1}
L. Dong and J. Rodrigo.
\newblock Finite-time singularities of an aggregation equation 
in {$\Bbb R^n$} with fractional dissipation.
\newblock {\em Comm. Math. Phys.}, {\bf 287}(2):687--703, 2009.

\bibitem{DongAG2}
L. Dong and J. Rodrigo.
\newblock Refined blowup criteria and nonsymmetric blowup of an aggregation equation.
\newblock {\em Adv. Math.}, {\bf 220}(6):1717--1738, 2009.

\bibitem{DongAG3}
L. Dong and Z. Xiaoyi.
\newblock On a nonlocal aggregation model with nonlinear diffusion.
\newblock {\em Discrete Contin. Dyn. Syst.}, {\bf 27}(1):301--323, 2010.

\bibitem{Friedlander}
S. Friedlander and V. Vicol.
\newblock Global well-posedness for an advection-diffusion equations 
arising in magneto-geostrophic dynamics.
\newblock {\em Preprint}, 2010.

\bibitem{Holden:book}
H. Holden, K.~H. Karlsen, K.-A. Lie, and N.~H. Risebro.
\newblock {\em Splitting for Partial Differential Equations with Rough
  Solutions. Analysis and Matlab programs.}
\newblock European Math. Soc. Publishing House, Z\"urich, 2010.

\bibitem{Holden:tao}
H. Holden, K.~H. Karlsen, N.~H. Risebro, and T.~Tao.
\newblock Operator splitting for the KdV equation.
\newblock {\em Math. Comp.},  to appear.

\bibitem{Holden:karper}
H.~Holden, K.~H. Karlsen, and T.~Karper.
\newblock Operator splitting for two-dimensional incompressible fluid
  equations.
\newblock {\em Math. Comp.}, to appear.

\bibitem{Holden:lubich}
H. Holden, C. Lubich, and N.~H. Risebro.
\newblock Operator splitting for partial differential equations with Burgers nonlinearity.
\newblock {\em Math. Comp.},  to appear.


\bibitem{Kiselev2}
A. Kiselev, F. Nazarov, and A. Volberg.
\newblock Global well-posedness for the critical 2{D} dissipative 
quasi-geostrophic equation.
\newblock {\em Invent. Math.}, {\bf 167}(3):445--453, 2007.

\bibitem{Kiselev3}
A. Kiselev, F. Nazarov, and R. Shterenberg.
\newblock Blow up and regularity for fractal {B}urgers equation.
\newblock {\em Dyn. Partial Differ. Equ.}, {\bf 5}(3):211--240, 2008.

\bibitem{Laurent}
T. Laurent.
\newblock Local and global existence for an aggregation equation.
\newblock {\em Comm. Partial Differential Equations},  {\bf 32}(10-12):1941--1964, 2007.


\bibitem{Majda:2002kx}
A.~J. Majda and A.~L. Bertozzi.
\newblock {\em Vorticity and incompressible flow}, volume~27 of {\em Cambridge
  Texts in Applied Mathematics}.
\newblock Cambridge University Press, Cambridge, 2002.

\bibitem{Mogilner1}
A. Mogilner and L. Edelstein-Keshet.
\newblock A non-local model for a swarm.
\newblock {\em J. Math. Biol.},  {\bf 38}(6):534--570, 1999.



\bibitem{Bertozzi1}
C. Topaz and A.~L. Bertozzi.
\newblock Swarming patterns in a two-dimensional kinematic 
model for biological groups.
\newblock {\em SIAM J. Appl. Math.},  {\bf 65}(1):152--174, 2004.

\bibitem{Bertozzi2}
C. Topaz, A.~L. Bertozzi, and M.~A. Lewis.
\newblock A nonlocal continuum model for biological aggregation.
\newblock {\em Bull. Math. Biol.},  {\bf 68}(7):1601--1623, 2006.

\end{thebibliography}
\end{document}